\documentclass[aos,noinfoline]{imsart}
\usepackage[ngerman,english]{babel}
\usepackage{amsthm}
\usepackage{bbm}
\usepackage{hyperref}
\usepackage{amssymb,amsmath}
\usepackage{natbib} 
\setattribute{aos}{aos}{My Journal}

\hypersetup{
  colorlinks   = true, 
  urlcolor     = blue, 
  linkcolor    = blue, 
  citecolor   = blue 
}
\newtheorem{satz}{Satz}[section]

\newtheorem{theorem}[satz]{Theorem}
\newtheorem{proposition}[satz]{Proposition}

\newtheorem{lemma}[satz]{Lemma}

\theoremstyle{definition}
\newtheorem{definition}[satz]{Definition}

\newtheorem{remark}[satz]{Remark}

\newtheorem{example}[satz]{Example}

\newtheoremstyle{nospace}
  {\topsep}
  {\topsep}
  {\itshape}
  {}
  {\bfseries}
  {}
  {0em}
  {}
\theoremstyle{nospace}
\newtheorem{assumption}[satz]{Assumption}

\newcommand{\1}{\mathbbm{1}}

\newcommand{\cov}{\text{Cov}}

\newcommand{\diag}{\text{diag}}

\newcommand{\E}{\mathbb E}
\newcommand{\ExpF}{\mathcal F_n}

\newcommand{\ExpG}{\mathcal G_n}
\newcommand{\ExpGC}{\mathcal G^c_n}
\newcommand{\ExpFS}{\mathcal F^s_n}
\newcommand{\ExpM}{\mathcal M_n}
\newcommand{\ExpMC}{\mathcal M^c_n}
\newcommand{\ExpMCTwo}{\mathcal M^{c'}_n}
\newcommand{\ExpMS}{\mathcal M^s_n}
\newcommand{\ExpMSloc}{\mathcal M^s_{n,\text{loc}}}
\newcommand{\ExpMSTwoloc}{\mathcal M^{s'}_{n,\text{loc}}}
\newcommand{\ExpMSS}{\mathcal S_n}
\newcommand{\ExpMSSTwo}{\mathcal S'_n}
\newcommand{\ExpMSSloc}{\mathcal S_{n,\text{loc}}}
\newcommand{\ExpMSSTwoloc}{\mathcal S'_{n,\text{loc}}}
\newcommand{\EEE}{\mathcal E}
\newcommand{\FFF}{\mathcal F}

\newcommand{\GGG}{\mathcal G}
\newcommand{\HHH}{\mathcal H}

\newcommand{\HBsym}{H^{\beta}_{\text{sym}}}
\newcommand{\Id}{\text{Id}}
\newcommand{\I}{\text{I}}
\newcommand{\III}{\mathcal I}

\newcommand{\IIn}{\mathcal I_n}

\newcommand{\IInp}{\mathcal I_{np}}
\newcommand{\IItr}{\mathcal I_{\pi_n}}
\newcommand{\Ltwosym}{L^2_{\text{sym}}}
\newcommand{\LLL}{\mathcal L}
\newcommand{\loc}{\text{loc}}

\newcommand{\NNN}{\mathcal N}
\newcommand{\N}{\mathbb N}
\newcommand{\OOO}{\mathcal O}

\newcommand{\R}{\mathbb R}
\newcommand{\Rddsym}{\R^{d\times d}_{\text{sym}}}

\newcommand{\Snon}{\Theta_1}
\newcommand{\Spar}{\Theta_0}

\newcommand{\tr}{\text{tr}}
\newcommand{\var}{\text{Var}}
\newcommand{\vc}{\text{vec}}
\newcommand{\XXX}{\mathcal X}

\newcommand{\ZZZ}{\mathcal Z}

\numberwithin{equation}{section}

\begin{document}

\begin{frontmatter}
\title{Asymptotic efficiency for covariance estimation under noise and asynchronicity}
\runtitle{Asymptotic efficiency for covariance estimation}
\begin{aug}
  \author{\fnms{Sebastian}  \snm{Holtz}\corref{}\thanksref{t2}\ead[label=e1]{holtz@math.hu-berlin.de}}

  \thankstext{t2}{Supported by the Deutsche Forschungsgemeinschaft via IRTG 1792 \textit{High Dimensional Nonstationary Time Series} and FOR 1735 \textit{Structural Inference in Statistics}.}

  \runauthor{S. Holtz}

  \affiliation{Humboldt-Universit\"at zu Berlin}

  \address{S. Holtz\\ Institut f\"ur Mathematik\\Humboldt-Universit\"at zu Berlin\\Unter den Linden 6\\10099 Berlin\\Germany\\ \printead{e1}}

\end{aug}

\begin{abstract}
:\ The estimation of the covariance structure from a discretely observed multivariate Gaussian process under asynchronicity and noise is analysed under high-frequency asymptotics. Asymptotic lower and upper bounds are established for a general Gaussian framework which provides benchmark cases for various Gaussian process models of interest. The parametric bounds give rise to infinite-dimensional convolution theorems for covariation estimation under asynchronicity, which is an essential estimation problem in finance.
\end{abstract}



\end{frontmatter}

\sloppy
\section{Introduction}
We study inference on scaling parameters of a conditionally Gaussian process under discrete noisy observations over a fixed time interval. There are still many open questions in the field of covariance estimation of Gaussian processes under high-frequency asymptotics. Existing results reveal surprising phenomena, such as unusual convergence rates and unexpected emergences of parameters in the asymptotic covariance of estimators, which calls for a better understanding of how the underlying signal process drives asymptotic quantities of interest. Particularly, the multidimensional interplay of estimation targets encumbers the understanding of central object, such as asymptotic information. Moreover, for covariance operators that depend on high-dimensional or possibly even infinite-dimensional parameters, the mathematical analysis is not trivial.

Gaussian processes constitute a versatile class with a wide range of applications. Finance marks a major field of interest in practice, where usually models driven by Brownian motions are regarded. Fractional processes yield a more controversial approach, cf. \citet{Rogers[1997]}, but are also highly relevant in, for example, geophysics and biomechanics, cf. \citet{Mandelbrot[1970]} and \citet{Bardet[2007]}. Integrated Gaussian processes are used in Physics and Biology, e.g. for modelling particles, cf. \citet{Tory[2000]}, or in the meteorological literature, cf. \citet{Boughton[1987]}. The increasing usage of sophisticated Gaussian processes, such as multifractional Brownian motions, cf. \citet{Bianchi[2013]}, calls for a general understanding of lower and upper bounds, at least for benchmark cases.

As mentioned conditionally Gaussian models play a major role in finance, where inference is commonly performed conditionally on the underlying volatility process, cf. \citet{Mykland[2012]} for a general framework. A fundamental estimation problem is the extraction of the quadratic covariation (or integrated covolatility) of a continuous martingale in terms of a Brownian motion under microstructure noise. Moreover, some even consider application-driven generalisations, such as asynchronous and irregular (non-equidistant) observation schemes with varying sample sizes. Several famous approaches exist, e.g. \citet{Zhang[2005]}, \citet{Jacod[2009]}, \citet{Barndorff[2011]}, \citet{Bibinger[2014]}, \citet{Hayashi[2005]} and \citet{Christensen[2013]}, with varying limiting behaviours depending on the employed estimation techniques. These variations make a comparison of the existing approaches difficult. Additionally and importantly, the asymptotic lower bounds are not yet completely understood, even under regular observation schemes. The reason for this lies in the fact that the underlying statistical properties in these models are mathematically highly involved, which can be seen by regarding the results on efficiency in the literature.

Notable works in the one-dimensional field exist, for a parametric set-up by \citet{Gloter[2001]}, and in a semi-parametric case by \citet{Reiss[2011]}, whose results are based on the verification of local asymptotic normality (LAN) and use sophisticated arguments such as asymptotic equivalences of experiments. An interesting finding in both cases, parametric and semi-parametric, is that due to the noise the optimal rate is of the unusual order $n^{-1/4}$. A multidimensional extension of these results marks the semi-parametric Cram\'er-Rao lower bound derived by \citet{Bibinger[2014]}. As the latter is provided under rather strong assumptions for synchronous and regular finite samples, in which non-parametric estimators are biased, an asymptotic characterisation of efficiency under asynchronicity is required. Moreover, \cite{Ogihara[2018]} derives asymptotic lower bounds for $d=2$.

Little is known about efficient estimation if the assumption that the signal is driven by a Brownian motion is dropped. The one-dimensional Cram\'er-Rao bound derived by \citet{Sabel[2014]} is noteworthy, where the signal is given by a fractional Brownian motion. However, an asymptotic and particularly multidimensional lower bound and its dependence on the Hurst parameter remain an open question.
\newpage
Estimation of scaling parameters of Gaussian processes under noise also attracts interest in other fields. Related models appear in nonparametric Bayesian problems, where Gaussian process priors subject to an unknown parameter (hyperparameter) are used, cf. \citet{Szabo[2013]}. The difference in their setting lies in the asymptotic behaviour of the scaling parameter itself, whose estimation is carried-out pathwise. Latent variance estimation can also be found in genetic fields, e.g. \citet{Verzelen[2018]}. Here, the task of estimating the heritability bears structural similarities to the problems in this work.

The aim of this paper is to provide a general asymptotic theory for Gaussian covariance estimation models. In the following Section~\ref{SecMainResults} the fundamental parametric model is introduced, in which the superposition of a scaled multivariate Gaussian process with additive errors is observed in equation \eqref{DiscPar}. A main contribution of this paper is the universal Convolution Theorem~\ref{ThmConvPara}, which gives a precise asymptotic characterisation of efficient estimation and includes the set-ups of \citet{Gloter[2001]} and \citet{Sabel[2014]} as special cases but also applies to more models of practical relevance given as examples below. Even though an idealised parametric model might not be as such utilisable for practical purposes, its asymptotic lower bounds provide a basic case benchmark for comparing estimation procedures of more general models. Moreover, the insight gained in the fundamental model might be used in far more complex models. This phenomenon resembles the approach with which the second main result, Theorem~\ref{ThmConvNonpara}, is derived, which marks a semi-parametric convolution theorem for estimating the integrated covolatility matrix. This result not only extends the set-up in \citet{Reiss[2011]} by multidimensionality and asynchronicity, but also weakens smoothness assumptions to Sobolev regularity $\beta>1/2$.

The following section gives an overview of the main results along with their proof techniques, imposed assumptions and examples. Section~\ref{SecParametric} contains the parametric analysis, particularly the verification of Theorem~\ref{ThmConvPara}. The construction of efficient estimators is followed by further asymptotic equivalences that provide further insight on the estimation problem. Section~\ref{SecSemiPara} concludes this work by the stepwise deduction of Theorem~\ref{ThmConvNonpara}. Most of the proofs and reviews of several mathematical concepts can be found in the \hyperref[Appendix]{Appendix}.

\newpage
\section{Methodology and main results}\label{SecMainResults}
\subsection{\textbf{Notation}}
We introduce spaces of matrix-valued functions as they appear as canonical parameter sets. For $A,B\in\R^{v\times w}$ and $C\in\R^{vw\times vw}$, let
\[\langle A,B\rangle_C:=\vc(A)^{\top}C\vc(B),\]
and set $\langle\cdot,\cdot\rangle:=\langle\cdot,\cdot\rangle_{I_{vw}}$, where $\vc(A)\in\R^{vw}$ is the vectorisation of $A$ and $I_k$ denotes the identity matrix in $\R^{k\times k}$. Denote the corresponding induced norms by $\|\cdot\|_C$ and $\|\cdot\|$, given that $C>0$, i.e., if $C$ is positive-definite. Note that $\|\cdot\|$ is just the Hilbert-Schmidt norm.

Further let for $u\in\N,\ \Omega:=[0,1]$ and $f,g:\Omega^u\to\R^{v\times w}$ the inner product
\[\langle f,g\rangle_{L^2}:=\int_{\Omega^u}\langle f(t),g(t)\rangle dt\]
induce the norm $\|\cdot\|_{L^2}$ and the space $L^2=L^2(\Omega^u,\R^{v\times w})$. For $\beta\in(0,2)$ the $L^2$-subspace $H^{\beta}=H^{\beta}(\Omega^u,\R^{v\times w})$ consists of all $f:\Omega^u\to\R^{v\times w}$ such that
\[\|f\|_{H^{\beta}}:=\sum_{k:|k|<\beta}\|f^{(k)}\|_{L^2}+|f|_{H^{\beta}}<\infty.\]
Here $|\cdot|_{H^{\beta}}$ denotes the Sobolev-Slobodeckij semi-norm given for $\beta\neq1$ by
\[|f|^2_{H^{\beta}}:=\sup_{k:|k|=\lfloor\beta\rfloor}\int_{\Omega^u}\int_{\Omega^u}\frac{\|f^{(k)}(x)-f^{(k)}(y)\|^2}{|x-y|^{2(\beta-\lfloor\beta\rfloor)+u}}dxdy,\]
where $\lfloor \beta\rfloor$ denotes the integer part of $\beta$, and by $\sum_{k:|k|=1}\|f^{(k)}\|^2_{L^2}$ otherwise, where $k\in\{0,1\}^u$ denotes a multiindex with $|k|=\sum^u_{i=1}k_i$. For $u=1$ we often write $f':=f^{(1)}$. Within $H^{\beta}$ the ball of radius $L>0$ is defined via
\[H^{\beta}_L:=\{f\in H^{\beta}:\|f\|_{H^{\beta}}\leq L\}.\]
For $\gamma\in(0,1]$ and $N>0$ H\"older balls are given by
\[C^{\gamma}_N:=\{f:\Omega\to\R:\sup\nolimits_{s,t\in\Omega}|f(s)-f(t)|^{\gamma}/|s-t|\leq N\}.\]
Symmetric co-domains $\R^{d\times d}_{\text{sym}}:=\{A\in\R^{d\times d}:A=A^{\top}\}$ are highlighted by the notation $\Ltwosym:=L^2(\Omega^u,\R^{d\times d}_{\text{sym}})$ and $\HBsym:=H^{\beta}(\Omega^u,\R^{d\times d}_{\text{sym}})$. It is a basic fact that if $\beta>u/2$ for any $f\in H^{\beta}(\Omega^u,\R^{v\times w})$ a continuous version can be obtained after possibly modifying $f$ on a zero-subset of $\Omega^u$. An overview over Sobolev spaces and their embedding properties with respect to H\"older spaces can be found in \citet{Triebel[2010]}.

For $Z\sim\NNN(0,I_d)$ the matrix $\ZZZ=\cov(\vc(ZZ^{\top}))$ is twice the so-called symmetriser matrix, i.e., it has the property $\ZZZ\vc(A)=\vc(A+A^{\top})$, $A\in\R^{d\times d}$, see e.g. \citet{Abadir[2005]}. Any $d^2\times d^2$-matrix $A\otimes A$ commutes with $\ZZZ$. Moreover, $\ZZZ$ is positive semi-definite and therefore not invertible.

For $(A_n)_{n\geq1}$ and $(B_n)_{n\geq1}$ in $\R^{d\times d}$ the expression $A_n\lesssim B_n$ means $\|A_n\|=\OOO(\|B_n\|)$ and $A_n\sim B_n$ means $A_n=\OOO(B_n)$ as well as $B_n=\OOO(A_n)$.

Finally, for a set of parameters $\Theta$ the Le Cam distance between two statistical experiments $\EEE=\{(X,\mathcal X,P_{\theta}):\theta\in\Theta\}$ and $\FFF=\{(Y,\mathcal Y, Q_{\theta}):\theta\in\Theta\}$ on Polish spaces is given by $\Delta(\EEE,\FFF):=\max\{\delta(\EEE,\FFF),\delta(\FFF,\EEE)\}$. Here $\delta$ denotes the one-sided deficiency
\[\delta(\EEE,\FFF):=\inf_K\sup_{\theta\in\Theta}\|K\cdot P_{\theta}-Q_{\theta}\|_{\text{TV}},\]
where the infimum is taken over all Markov kernels from $(X,\mathcal X)$ to $(Y,\mathcal Y)$ and $\|\cdot\|_{\text{TV}}$ denotes the total variation norm. Sequences $(\EEE_n)_{n\geq1}$ and $(\FFF_n)_{n\geq1}$ of experiments are called asymptotically equivalent if $\Delta(\EEE_n,\FFF_n)=o(1)$. The latter implies that asymptotic properties transfer from one model to the other, and vice versa. Properties of $\Delta$ can be found in Appendix~\ref{SsecLeCamDistance} and \ref{SsecWeakLanReg}, see also \cite{LeCam[2000]} for a thorough introduction.
\subsection{\textbf{Fundamental parametric model}}
Consider the $d$-dimensional discrete observation model generated by the observations
\begin{equation}\label{DiscPar}
\tilde Y_i=\Sigma^{1/2}G_{i/n}+\varepsilon_i,\quad i=1,\ldots,n,
\end{equation}
where $G=(G_t)_{t\in[0,1]}$ is such that $G\sim\NNN^{\otimes d}_{0,\Gamma}$, for a centred Gaussian measure $\NNN_{0,\Gamma}$ on $L^2([0,1],\R)$ with covariance operator $\Gamma$. Assume that $G$ is independent of the i.i.d. errors $\varepsilon_1,\ldots,\varepsilon_n\sim\NNN(0,\eta^2I_d)$. The noise level $\eta>0$ is a nuisance parameter, whereas $\Sigma$ is the parameter of interest subject to
\begin{equation}\label{Parspace}
\Spar:=\{\Sigma\in\R^{d\times d}_+:0<\Sigma<SI_d\},
\end{equation}
where $S>0$. Here $\R^{d\times d}_+$ denotes all positive-definite $\R^{d\times d}$-matrices and the ordering $\Sigma<SI_d$ is meant with respect to positive definiteness.

An important tool paving the way to asymptotic lower bounds in the present work are several asymptotic equivalences in Le Cam's sense. In order to obtain a mathematically more convenient working basis, consider the spectral analogue of \eqref{DiscPar} given by
\begin{equation}\label{SeqPar}
Y_p\sim\NNN(0,C_p),\quad C_p:=\Sigma\lambda_p+\frac{\eta^2}{n}I_d,\quad p\geq1.
\end{equation}
The sequence $\lambda=(\lambda_p)_{p\geq1}$ denotes the eigenvalue sequence of the covariance operator of $\Gamma$. The approximation error between the models \eqref{DiscPar} and \eqref{SeqPar} is quantifiable by the Le Cam $\Delta$-distance, which is negligible under the following regularity assumption, cf. Proposition~\ref{PropFDiscSeq} below.
\begin{assumption}\label{AssSoboCov}\textbf{-}$\pmb{G(\beta).}$ The function $(s,t)\mapsto\cov(G_s,G_t),\ s,t\in[0,1]$, lies in $H^{\beta}$ for some $\beta\in(1,2)$.
\end{assumption}

As an important consequence of asymptotic equivalence, LAN-expansions and convolution theorems in \eqref{DiscPar} and \eqref{SeqPar} coincide. However, as there are infinitely many non-identically distributed vectors $Y_p$ in \eqref{SeqPar} it is not clear at all whether a LAN-expansion holds since the sum of infinitely many remainder terms needs to be controlled. For the latter it will be crucial that the behaviour of certain subsequences $(\lambda_{p_n})_{n\geq1}$ carries over to the entire sequence $(\lambda_p)_{p\geq1}$ which can be done under the following.
\begin{assumption}\label{regvarass}\textbf{-}$\pmb{\lambda(\delta).}$ The eigenvalues $\lambda=(\lambda_p)_{p\geq1}$ of $\Gamma$ are strictly-positive and regularly varying at infinity with index $-\delta,\ \delta>1$, i.e.,
\begin{equation}\label{regvar}
\lim_{p\to\infty}\frac{\lambda_{\lfloor ap\rfloor}}{\lambda_p}=a^{-\delta},\ \forall a>0.
\end{equation}
\end{assumption}

If $P^n_{\Sigma}$ denotes the measure induced by \eqref{SeqPar} then Assumption~\ref{regvarass}-$\lambda(\delta)$ ensures that a certain LAN-expansion holds, i.e., for $H\in\Rddsym$ one has
\[\log\frac{dP^n_{\Sigma+r_nH}}{dP^n_{\Sigma}}\overset{P^n_{\Sigma}}{\to}\Delta_H-\frac{1}{2}\|H\|^2_{\III(\Sigma)\ZZZ},\]
where $\Delta_H\sim\NNN(0,\|H\|^2_{\III(\Sigma)\ZZZ})$ and $\III(\Sigma)\ZZZ\in\R^{d^2\times d^2}$ is the asymptotic Fisher information matrix, cf. Proposition~\ref{LAN_gen}. The rate $r_n\to0$ is obtained by
\[\lim_{n\to\infty}n\lambda_{\lfloor r^{-2}_n\rfloor}=c,\]
where $c>0$ is chosen such that $r_n$ is normalised with respect to multiplicative scalars, e.g. $r_n=n^{-1/4}$ but not $r_n=2n^{-1/4}$. Thus a slow decay of $\lambda$ implies a fast decay of $r_n$, and vice versa. Since the Fisher information $\III(\Sigma)\ZZZ$ is singular it is not obvious how classical implications from LAN-theory, e.g. a convolution theorem, can be obtained. This problem is overcome by symmetrising properties of $\ZZZ$ which allow for certain isometries, cf. Remark~\ref{RmkZZZ} below. In a non-noisy set-up \citet{Brouste[2018]} recently derived asymptotic lower bounds despite singularity by usage of certain rate matrices. For a further discussion of $r_n$ and $\III(\Sigma)$ see Section~\ref{efficiency}.

\subsection{\textbf{Parametric main result}} Let $\psi(\Sigma)\in\R^k$ be a differentiable target of estimation in the sense that there is some $\nabla\psi_{\Sigma}\in\R^{k\times d^2}$ such that
\begin{equation}\label{reg_esti}
r^{-1}_n(\psi(\Sigma+r_nH)-\psi(\Sigma))\to\nabla\psi_{\Sigma}\vc(H),\quad H\in\Rddsym,
\end{equation}
as $n\to\infty$. In the following, sequences of so-called regular estimators $\hat\vartheta_n$ of $\psi(\Sigma)$ are regarded, cf. Appendix~\ref{SsecWeakLanReg} for a definition.
\begin{theorem}\label{ThmConvPara}
Let $\hat\vartheta_n$ be a sequence of regular estimators of $\psi(\Sigma)\in\R^k$ with \eqref{reg_esti} and suppose that Assumptions~\ref{AssSoboCov}-$G(\beta)$ and \ref{regvarass}-$\lambda(\delta)$ are met. Then under $P^n_{\Sigma+r_nH},\ H\in\Rddsym$, and as $n\to\infty$ it holds that
\[r^{-1}_n(\hat\vartheta_n-\psi(\Sigma+r_nH))\overset{d}{\to}\NNN\left(0,\tfrac{1}{4}\nabla\psi^{\top}_{\Sigma}\III(\Sigma)^{-1}\ZZZ\nabla\psi_{\Sigma}\right)\ast R,\]
for some distribution $R$.
\end{theorem}
The deduction of the above result offers a comprehensive understanding of how efficient estimation, particularly the optimal estimation rate $r_n$ and the geometry of the Fisher information matrix, depends on the spectral properties of the signal. Moreover, Theorem~\ref{ThmConvPara} extends the knowledge of asymptotic lower bounds in a few one-dimensional models to a general class of underlying multidimensional Gaussian processes. It is noted that only the leading term of $(\lambda_p)_{p\geq1}$ has to be known for the derivation of lower bounds.

As mentioned before, several estimators have been designed for particular Gaussian models. In this work a universal estimation approach is given by
\[\hat\vartheta^{\text{ad}}_n:=\sum_{p\in\pi_n}W_p\lambda^{-1}_p\vc(Y_pY^{\top}_p-\eta^2/n I_d),\]
where $\pi_n\subsetneq\N$ and $W_p\in\R^{d^2\times d^2}$ are adaptive weights. A spectral approach has been already used, e.g. by \citet{Bibinger[2014]}, for a covariation estimator, where martingale properties inherited from the Brownian motion are a key argument. In contrary, constructing $W_p$ independently of $(Y_p)_{p\in\pi_n}$ is the crucial idea in this work, which yields generality and gives
\[r^{-1}_n(\hat\vartheta^{\text{ad}}_n-\psi(\Sigma+r_nH))\overset{d}{\to}\NNN\left(0,\tfrac{1}{4}\nabla\psi^{\top}_{\Sigma}\III(\Sigma)^{-1}\ZZZ\nabla\psi_{\Sigma}\right),\]
under $P^n_{\Sigma+r_nH}$, for any $H\in\Rddsym$, cf. Theorem~\ref{ThmOrAd}. The matching upper bounds imply that the derived lower bounds from Theorem~\ref{ThmConvPara} are sharp.
\begin{remark}\label{RmkDepNoise}
If the model is generalised to non-diagonal noise $\varepsilon_1,\ldots,\varepsilon_n\sim\NNN(0,H)$ with $H\in\R^{d\times d}_+$ known, then lower and upper bounds can be derived in the same way if the transformations $\tilde Y'_i:=H^{-1/2}Y_i$, $i=1,\ldots,n$ are used. In particular, $\Sigma$ in $\III(\Sigma)$ has to be replaced by $H^{-1/2}\Sigma H^{-1/2}$ and $\eta^2$ is set to the value $1$.
\end{remark}
\begin{remark}\label{RmkWeakDep}
Another possible extension is given by weakly dependent noise. Let us consider stationary $m$-dependent noise, i.e., $\E[\varepsilon_i\varepsilon_{i+j}]=\eta_j$ with $\eta_j=0,\ j>m$, which is used in high-frequency statistics, e.g. by \citet{Hautsch[2013]}. With $\eta'_n:=\var(n^{-1/2}\sum^n_{i=1}\varepsilon_i)=\eta_0+2\sum^m_{j=1}\frac{n-j}{n}\eta_j$ a `big-block-small-block' argument gives rise to the desired connection between discrete and sequence space model in the sense that $\eta^2$ in \eqref{SeqPar} should be replaced with $\lim_{n\to\infty}\eta'_n$ and the theory provided by this work can be applied. However, this results in more assumptions on $\beta$, $\gamma$ and $m$ and is therefore omitted.
\end{remark}
\begin{remark}\label{RmkSigmaRandom}
The techniques of this work can also be carried out if $\Sigma$ is random but $G$ given $\Sigma$ is still Gaussian. The derivation of a conditional convolution theorem is then obtained if Assumption H0 (which replaces the usage of Le Cam's third Lemma) of the general result by \citet{Clement[2013]} is met. Again, precise derivations are omitted.
\end{remark}

\begin{example}\label{ExBM}
If $G$ denotes a $d$-dimensional Brownian motion, then $\lambda^{\text{BM}}_p=(\pi(p-1/2))^{-2}$, i.e., Assumption~\ref{regvarass} holds with $\delta=2$. Then efficient regular estimators $\hat\vartheta_n$ of $\vartheta=\vc(\Sigma)$ satisfy (cf. Theorem~\ref{ThmRateFisher} below)
\begin{equation}\label{clt_bm}
n^{1/4}(\hat\vartheta_n-\vartheta)\overset{d}{\to}\NNN(0,2\eta(\Sigma\otimes\Sigma^{1/2}+\Sigma^{1/2}\otimes\Sigma)\ZZZ).
\end{equation}
For $d=1$ this result coincides with \citet{Gloter[2001]} and for $d\geq 1$, \eqref{clt_bm} extends asymptotically the Cram\'er-Rao bound of \citet{Bibinger[2014]}.
\end{example}

\begin{example}\label{ExfBM}
If $G$ is a fractional Brownian motion with Hurst exponent $H\in(0,1)$, then, by \citet{Chigansky[2018]}, the corresponding eigenvalues satisfy \eqref{regvar} with $\delta=2H+1$:
\[\lambda^{\text{fBM}}_p=\frac{\sin(H\pi)\Gamma(2H+1)}{(\pi p)^{2H+1}}+o(p^{-(2H+1)}),\quad p\geq1.\]
Precise asymptotic lower bounds have only been known for $d=1$ in a non-noisy setting, cf. \citet{Brouste[2018]}. In the multivariate noisy set-up Theorem~\ref{ThmConvPara} implies for $H>1/4$ that the rate of of efficient estimators is $r_n=n^{-1/(4H+2)}$, where the restriction $H>1/4$ ensures Assumption~\ref{AssSoboCov}-$G(\beta)$. The optimal asymptotic covariance can be easily calculated by Theorem~\ref{ThmRateFisher} below. Note that the Cram\'er-Rao bound in \citet{Sabel[2014]} holds for any $H\in(0,1)$. Whether the models \eqref{DiscPar} and \eqref{SeqPar} can be separated for $H\leq 1/4$ lies beyond the scope of this paper.
\end{example}

\begin{example}
The eigenvalues $\lambda^{\text{BB}}_p=(\pi p)^{-2}$ corresponding to a Brownian bridge have the same leading term as $\lambda^{\text{BM}}_p$ in Example~\ref{ExBM}, hence \eqref{clt_bm} holds as well. Similarly, regard the (stationary) Ornstein-Uhlenbeck process
\[\Sigma^{1/2}G_t=\Sigma^{1/2}G_0e^{-\beta t}+\Sigma^{1/2}\int^t_0e^{-\beta(t-s)}dB_s,\quad t\in[0,1],\]
where $G_0\sim\NNN(0,(2\beta)^{-1}I_d),\ \beta>0$ and $B$ is a standard Brownian motion. Under the normalisation $\beta=1/2$ the eigenvalues $\lambda^{\text{OU}}_p=\frac{2\beta}{p^2\pi^2}+o(p^{-2})$ imply \eqref{clt_bm} as well. This means that mean-reversion or the behaviour of bridges have no impact on estimation of $\Sigma$. In fact, the three models corresponding to $\lambda^{\text{BM}}_p$, $\lambda^{\text{BB}}_p$ and $\lambda^{\text{OU}}_p$ are even asymptotically equivalent, cf. Proposition~\ref{PropEigLeCamEqui}.

Similarly a fractional Brownian bridge and a fractional Ornstein-Uhlenbeck process seem to offer the same asymptotics as $\lambda^{\text{fBM}}_p$, cf. the (yet unpublished) drafts by \citet{Chigansky[2017]} and \citet{Chigansky[2018b]}.
\end{example}

\begin{example}
For the $m$-fold integrated Brownian motion the eigenvalues satisfy $\lambda^{m\text{BM}}_p=(\pi p)^{-(2m+2)}+o(p^{-(2m+2)})$, cf. \citet{Wang[2008]}. This implies $r_n=n^{-1/(4m+4)}$, which reveals the interesting phenomenon that very smooth signal paths lead to rather poor estimation rates, also cf. Example~\ref{ExfBM}, where regularity is increasing in $H$ whereas $r_n$ is decreasing.
\end{example}

\subsection{\textbf{Semi-parametric asynchronous model}}
On the basis of the parametric results asymptotic lower bounds in the more sophisticated asynchronous observation model
\begin{equation}\label{DiscSemi}
Y_{i,j}=(X_{t_{i,j}})_j+\varepsilon_{i,j},\quad 1\leq i\leq n_j,\ 1\leq j\leq d,
\end{equation}
are derived, where $X_t=X_0+\int^t_0\Sigma^{1/2}(s)dB_s$ denotes a continuous martingale in terms of a $d$-dimensional standard Brownian motion $B=(B_t)_{t\in[0,1]}$. The noise variables $\varepsilon_{i,j}\sim\NNN(0,\eta^2_j),\ 1\leq i\leq n_j$, with $\eta_j>0$ known, $1\leq j\leq d$, are mutually independent and independent of the signal $X=(X_t)_{t\in[0,1]}$. Moreover, suppose for the asymptotics $n_{\min}:=\min_{1\leq j\leq d}n_j\to\infty$ that $n_{\min}/n_j\to\nu_j$ for some $\nu_j\in(0,1],\ j=1,\ldots,d$.
\begin{assumption}\label{AssSigma}\textbf{-}$\pmb{\Sigma(\beta,M,S).}$ For some $\beta>1/2$, $M>0$, and $S>1$ we assume that $\Sigma$ belongs to the parameter set
\[\Snon:=\left\{A:[0,1]\to\Rddsym\Big|A\in H^{\beta}_M:S^{-1}I_d<A(t)<SI_d,\forall t\in[0,1]\right\}.\]
\end{assumption}
\begin{assumption}\label{AssRegF}\textbf{-}$\pmb{F(\gamma,N,\beta).}$ The observation times obey $t_{i,j}=F^{-1}_j(i/n_j)$ for a distribution function $F_j:[0,1]\to[0,1]$ with derivative $F'_j$ and
\begin{itemize}
\item[(i)] $F_j(0)=0$ and $F_j(1)=1$,
\item[(ii)] $F'_j\in C^{\gamma}_N$ and $F'_j>0$,
\end{itemize}
for $j=1,\ldots,d$, and some $\gamma\in(\beta,1],\ N>0$.
\end{assumption}

As in the parametric set-up, \eqref{DiscSemi} is approximated by a spectral representation for which the conditions $\gamma>\beta>1/2$ and $\Sigma>S^{-1}I_d$ are needed. The latter one is slightly restrictive but not uncommon, cf. \citet{Reiss[2011]}. The spectral representation is given by the mutually independent random vectors
\begin{equation}\label{SeqSemi}
Y_{pk}\sim\NNN(0,C_{pk}),\quad k=0,\ldots,m-1,\ p\geq 1,
\end{equation}
where $C_{pk}:=\Sigma(k/m)\lambda_{mp}+n^{-1}_{\min}\Xi^2(k/m),\ \lambda_{mp}:=(\pi pm)^{-2}$ and
\[\Xi^2(t):=\diag(\eta^2_j\nu_j/(F'_j(t)))_{1\leq j\leq d}.\]
However, the approximation of \eqref{DiscSemi} by \eqref{SeqSemi} holds only for localisations $\Sigma+n^{-1/4}_{\min}H,\ H\in\HBsym$, which nevertheless is the right ingredient to ensure that LAN-expansions in the sequence space carry over to \eqref{DiscSemi}, cf Proposition~\ref{PropLANequi}.

\subsection{\textbf{Semi-parametric main result}} For each $k$ the sequence $(Y_{pk})_{p\geq1}$ in \eqref{SeqSemi} is of the same type as the fundamental sequence space model in \eqref{SeqPar}. Indeed the parametric results can be applied simultaneously (over $k$) to the setting \eqref{SeqSemi}, for which we consider targets of estimation given by
\begin{equation}\label{targetinf}
\psi(\Sigma):=\int^1_0(W(\Sigma))(t)dt
\end{equation}
with a differentiable weight $W:\Theta_1\to L^2([0,1],\R^{d^2})$ in the sense that
\begin{equation}\label{targetreg}
n^{1/4}_{\min}(W(\Sigma+n^{-1/4}_{\min}H)-W(\Sigma))\to\nabla W_{\Sigma}\cdot\vc(H),\quad H\in\HBsym,
\end{equation}
as $n_{\min}\to\infty$, for some $\nabla W_{\cdot}\in L^2([0,1],\R^{d^2\times d^2})$. An example is given by the choice $W(\Sigma)=\vc(\Sigma)$ with $\nabla W_{\cdot}=I_{d^2}$.
\begin{theorem}\label{ThmConvNonpara}
Let $\hat\vartheta_n$ be a sequence of regular estimators of $\psi(\Sigma)$ as in \eqref{targetinf} with \eqref{targetreg} and suppose that Assumptions~\ref{AssSigma}-$\Sigma(\beta,M,S)$ and \ref{AssRegF}-$F(\gamma,N,\beta)$ are met. Then under $Q^n_{\Sigma+n^{-1/4}_{\min}H},\ H\in\HBsym$, it holds that
\[n^{1/4}_{\min}(\hat\vartheta_n-\psi(\Sigma+n^{-1/4}_{\min}H))\overset{d}{\to}\NNN\Big(0,\frac{1}{4}\int^1_0(\nabla W_{\Sigma}\III^{-1}_{\Sigma}\ZZZ\nabla W_{\Sigma}^{\top})(t)dt\Big)\ast R,\]
as $n_{\min}\to\infty$, for some $R$, where $Q^n_{\Sigma}$ is the measure induced by \eqref{SeqSemi} and
\[\III^{-1}_{\Sigma}(t)=8(\Sigma^{1/2}_{\Xi}(t)\otimes\Sigma(t)+\Sigma(t)\otimes \Sigma^{1/2}_{\Xi}(t)),\ t\in[0,1],\]
with $\Sigma^{1/2}_{\Xi}:=\Xi(\Xi^{-1}\Sigma\Xi^{-1})^{1/2}\Xi$.
\end{theorem}
The above statement extends the one-dimensional asymptotic efficiency results of \citet{Reiss[2011]} in various ways. Firstly, the needed H\"older-regularity $(1+\sqrt{5})/4\approx0.81$ in \citet{Reiss[2011]} can be relaxed to Sobolev regularity $\beta>1/2$. This relaxation is achieved by focussing on asymptotically equivalent experiments that share the same semi-parametric lower bounds for targets as in \eqref{targetinf}, whereas Rei\ss\ even considers experiments with common asymptotic non-parametric lower bounds. Moreover, Theorem~\ref{ThmConvNonpara} allows for multidimensionality of $\Sigma$ as well as for asynchronicity and therefore extends asymptotically the basic case Cram\'er-Rao bound for continuously differentiable $\Sigma$ by \citet{Bibinger[2014]}. Since the local method of moments estimator provided by \citet{Bibinger[2014]} attains the Gaussian part of the limit distribution of Theorem~\ref{ThmConvNonpara}, the derived bounds are sharp.

\begin{remark}
The steps that are taken to establish Theorem~\ref{ThmConvNonpara} can be developed analogously if $\Sigma=(\Sigma_t)_{t\in[0,1]}$ is assumed to be random with realisations in $\Snon$ and if $X$ conditioned on $\Sigma$ is still Gaussian. Again the result by \citet{Clement[2013]} gives a conditional convolution theorem, cf. Remark~\ref{RmkSigmaRandom}. The estimator provided by \citet{Altmeyer[2015]} attains the corresponding asymptotic stochastic lower bounds. Similarly, extensions for the noise can be obtained as illustrated in Remark~\ref{RmkDepNoise} and \ref{RmkWeakDep}.
\end{remark}

\section{Analysis of the fundamental parametric model}\label{SecParametric}
Throughout this section we assume that $\Sigma\in\Spar$ for some $S>0$, cf. \eqref{Parspace}, and that Assumption~\ref{AssSoboCov}-$G(\beta)$ and Assumption~\ref{regvarass}-$\lambda(\delta)$ are satisfied.
\subsection{\textbf{Connection between discrete and sequence space model}}\label{secCtsSeq}
Consider the discrete observation model \eqref{DiscPar} and its continuous analogue
\begin{equation}\label{ContPar}
dY_t=\Sigma^{1/2}G_tdt+\frac{\eta}{\sqrt{n}}dW_t,\quad t\in[0,1],
\end{equation}
where $W$ is a Wiener process independent of $G$. The model \eqref{ContPar} is consistent with observing the stochastic bilinear forms
\begin{equation}\label{cylmeas}
Y_f:=(f,dY):=\sum^d_{j=1}\int^1_0(f(t))_jd(Y_t)_j,\quad f\in L^2([0,1],\R^d).
\end{equation}
$Y_f$ is Gaussian with $\E[Y_f]=0$ and $\text{Cov}(Y_f,Y_g)=\langle K_{\Sigma,n}f,g\rangle_{L^2}$. The underlying covariance operator $K_{\Sigma,n}$ is given by
\[K_{\Sigma,n}:=T_{\Sigma^{1/2}}\diag(\Gamma)_{1\leq j\leq d}T_{\Sigma^{1/2}}+\frac{\eta^2}{n}\Id,\]
with $T_{\Sigma^{1/2}}:f\mapsto\Sigma^{1/2} f,\ \Id:f\mapsto f$ and $\diag(\Gamma)_{1\leq j\leq d}:f\mapsto(\Gamma f_j)_{1\leq j\leq d}$ being the covariance operator of $G$. For the orthonormal eigenbasis $(\varphi_p)_{p\geq1}$ of $\Gamma$ and $e_{pi}:=(\1_{\{i=j\}}\varphi_p)_{1\leq j\leq d}$ the vectors $(Y_{e_{p1}},\ldots,Y_{e_{pd}})^{\top},\ p\geq1,$ follow the same distribution as the sequence $(Y_p)_{p\geq1}$ in \eqref{SeqPar}.
\begin{definition}
Denote by $\ExpF$ and $\ExpFS$ the statistical experiments that are generated by the observations \eqref{DiscPar} and \eqref{SeqPar}, respectively.
\end{definition}
Since $(\varphi_p)_{p\geq 1}$ is a basis, observing the sequence $(Y_p)_{p\geq1}$ in \eqref{SeqPar} is equivalent to observe \eqref{ContPar}. Moreover, the following is just a consequence of the more general Theorem~\ref{ThmLeCamDiscCont} given in the Appendix.
\begin{proposition}\label{PropFDiscSeq}
Under Assumption~\ref{AssSoboCov}-$G(\beta)$ the experiments $\ExpF$ and $\ExpFS$ are asymptotically equivalent. More precisely, the Le Cam distance obeys
\[\Delta(\ExpF,\ExpFS)=\OOO(Sn^{1-\beta}).\]
\end{proposition}

\subsection{\textbf{Local asymptotic normality}}\label{efficiency}
Denote the score in $\ExpFS$ by $\nabla\ell_n(\Sigma):=\sum_{p\geq1}\ell_{np}(\Sigma)$ and set $\IIn(\Sigma)\ZZZ:=\cov(\nabla\ell_n(\Sigma))$, where
\begin{equation}\label{SeqScore}
\nabla\ell_{np}(\Sigma):=\frac{1}{2}\lambda_p\vc(C^{-1}_pY_{p}Y^{\top}_{p}C^{-1}_p-C^{-1}_p).
\end{equation}
The Fisher information $\IIn(\Sigma)\ZZZ=\sum_{p\geq1}\IInp(\Sigma)\ZZZ\in\R^{d^2\times d^2}$ is driven by
\[\IInp(\Sigma):=\frac{1}{4}\lambda^2_p(C^{-1}_p\otimes C^{-1}_p)\quad p\geq1.\]
In the derivation of $\ell_n$ and $\III_n$ the following well-known identity was used:
\[\vc(ABC)=(C^{\top}\otimes A)\vc(B),\quad A,B,C\in\R^{d\times d}.\]

As a consequence of Assumption~\ref{regvarass}, $\IIn(\Sigma)\ZZZ$ is well-defined. A crucial quantity is the rate $r_n\to0$ such that the asymptotic Fisher information
\[\III(\Sigma)\ZZZ:=\lim_{n\to\infty}r^2_n\IIn(\Sigma)\ZZZ\]
is well-defined, where $r_n$ is assumed to be normalised with respect to scalars, e.g. $r_n=n^{-1/4}$. The key to finding this rate $r_n$ lies in the interplay between the operators $\diag(\Gamma)_{1\leq j\leq d}$ and $\tfrac{1}{n}\Id$ along with the regular variation of $\lambda$. More precisely, in the covariance matrices $C_p=\Sigma\lambda_p+\tfrac{\eta^2}{n}I_d$, the impact of signal and noise is (nearly) balanced at the index $p_n$ with $\lambda(p_n)=n^{-1}$, where we identify the sequence $\lambda$ with some continuously interpolated non-increasing analogue $\lambda:\R_+\to\R_+$. It is well-known, that the representation
\begin{equation}\label{regvarrepr}
\lambda(p)=p^{-\delta}L(p)
\end{equation}
is valid, for some slowly varying $L:\R_+\to\R_+$, cf. \citet{Bingham[1989]}.

\begin{theorem}\label{ThmRateFisher}
Grant Assumption~\ref{regvarass}-$\lambda(\delta)$ on $\Gamma$. Then the Fisher information satisfies for any $\Sigma\in\Rddsym$ with $\Sigma>0$
\begin{equation}\label{FisherRate}
p^{-1}_n\IIn(\Sigma)\ZZZ\to\III(\Sigma)\ZZZ,\quad\text{as }n\to\infty,
\end{equation}
where $p_n$ is given by $\lambda(p_n)=n^{-1}$. If $Q$ is an orthogonal matrix such that $\Sigma=Q^{\top}\diag(s_1,\ldots,s_d)Q$ then
\[\III(\Sigma)=(Q\otimes Q)^{\top}\diag(v_{11},\ldots,v_{1d},v_{21},\ldots,v_{2d},v_{31},\ldots,v_{dd})(Q\otimes Q)\]
with eigenvalues
\[v_{i,j}=\frac{\zeta}{4\eta^{2/\delta}}\int^1_0(s_i+x^{\delta})^{-1}(s_j+x^{\delta})^{-1}dx,\quad i,j=1,\ldots,d,\]
where $\zeta=\lim_{n\to\infty}r^2_np_n$ for $r_n\sim p^{-1/2}_n$ standardised. Moreover, the convergence in \eqref{FisherRate} already holds for $\III_{\pi_n}(\Sigma):=\sum_{p\in\pi_n}\IInp(\Sigma)$, whenever $\pi_n=[\underline{\pi_n},\overline{\pi_n}]\cap\N,$ with $\underline{\pi_n}/p_n\to0$ and $(\underline\pi_n\wedge\overline{\pi_n}/p_n)\to\infty$.
\end{theorem}
By the above statement the rate $r_n$ satisfies the relation
\[r_n L(r^{-2}_n)^{1/(2\delta)}\sim n^{-1/(2\delta)},\]
with $L$ as in \eqref{regvarrepr}. Thus the rate $r_n$ is completely determined by the decay of $\lambda$. The slower $\lambda$ decreases the more observations $Y_p$ carry significant information about $\Sigma$ and the faster $\Sigma$ can be estimated. Moreover, solely the limiting behaviour of $L$ determines the constant $\zeta$. For instance, in the Brownian motion case $\lambda^{\text{BM}}_p=(p-1/2)^{-2}\pi^{-2}$ one has $\delta=2,p_n=\sqrt{n}/\pi+1/2$ and $L(p)=(\pi(2-1/(2p)))^{-2}$, which gives $r_n=n^{-1/4}$ and $\zeta=1/\pi$.

A simple calculation, cf. Remark~\ref{RmkFisherExp}, shows, that the eigenvalues obey
\[v_{i,j}=\frac{\zeta\pi}{4\delta\sin(\pi/\delta)\eta^{2/\delta}}\cdot\frac{s^{1/\delta-1}_j-s^{1/\delta-1}_i}{s_i-s_j}\]
and that they are driven by the slope of $x\mapsto-x^{1/\delta-1}$ between all pairs $(s_i,s_j)$. Whenever $s_i=s_j$ the slope equals the derivative at $s_i$. In particular, for the case $\Sigma=\sigma^2\in\R_+$ the Fisher information becomes
\[\III(\sigma^2)=\frac{\zeta\pi(1-1/\delta)}{4\delta\sin(\pi/\delta)\eta^{2/\delta}}\sigma^{2/\delta-4}.\]

Sufficient information to estimate $\Sigma$ efficiently in asymptotics is already provided by those observations $Y_p$ in $\ExpFS$, such that $p$ is subject to an interval $\pi_n$ as in Theorem~\ref{ThmRateFisher}. This means that maximal information about $\Sigma$ is asymptotically contained in (arbitrarily slowly) increasing neighbourhoods of $p_n$ within the spectrum of $Y=(Y_t)_{t\in[0,1]}$ in $\ExpF$. This gives canonical choices of truncation indices for spectral estimators of $\Sigma$, cf. Section~\ref{estimation}.

For $\Sigma\in\Spar$ consider local alternatives of the form $\Sigma+r_nH,\ H\in\Rddsym$, where $r_n$ is chosen according to Theorem~\ref{ThmRateFisher}. Note that $\Sigma+r_nH\in\Spar$ for $n$ sufficiently large, hence $P^n_{\Sigma+r_nH}$ might be defined arbitrarily, whenever $\Sigma+r_nH\notin\Spar$. Denote by $\Delta_H$ the centred Gaussian process with
\[\cov(\Delta_{H_1},\Delta_{H_2})=\langle H_1,H_2\rangle_{\III(\Sigma)\ZZZ},\quad H_1,H_2\in\Rddsym,\]
where it is noted that $\III(\Sigma)\ZZZ$ is positive definite on $\{\vc(H):H\in\Rddsym\}$.

\begin{proposition}\label{LAN_gen}
Under Assumption~\ref{regvarass}-$\lambda(\delta)$, for any $\Sigma\in\Spar$, the following asymptotic expansion is satisfied in $\ExpFS$ as $n\to\infty$:
\begin{equation}\label{LAN}
\log\frac{dP^n_{\Sigma+r_nH}}{dP^n_{\Sigma}}=\Delta_{n,H}-\frac{r^2_n}{2}\|H\|^2_{\IIn(\Sigma)\ZZZ}+\rho_n,\quad H\in\Rddsym,
\end{equation}
where $\Delta_{n,H}\overset{d}{\to}\Delta_H$, under $P^n_{\Sigma},\ r^2_n\|H\|^2_{\IIn(\Sigma)\ZZZ}\to\| H\|^2_{\III(\Sigma)\ZZZ}$ and $\rho_n=o_{P^n_{\Sigma}}(1)$.
\end{proposition}
Note that $\Delta_{n,H}=r_n\vc(H)^{\top}\nabla\ell_n(\Sigma)$, where $\nabla\ell_n$ denotes the score in $\ExpFS$. Moreover, the remainder obeys $\rho_n=\rho^{(1)}_n+\rho^{(2)}_n$ with $\E[\rho^{(1)}_n]=0$ and
\begin{align}
\label{LANrem1}\var(\rho^{(1)}_n)\leq&r^2_n\|H\|^2\|\Sigma^{-1}\|^2r^2_n\|H\|^2_{\IIn(\Sigma)\ZZZ}=\OOO(r^2_n),\\
\label{LANrem2}|\rho^{(2)}_n|\leq&2r_n\|H\|\|\Sigma^{-1}\|r^2_n\|H\|^2_{\IIn(\Sigma)\ZZZ}=\OOO(r_n),
\end{align}
hence \eqref{LAN} holds uniformly in $H$ over balls within $\Rddsym$.\newpage

An implication of the LAN-property~\eqref{LAN} is weak convergence of the localisations $\{P^n_{\Sigma+r_nH}:H\in\Rddsym\}$ to the Gaussian shift experiment $\GGG:=\{\NNN(\III(\Sigma)\ZZZ\vc(H),\III(\Sigma)\ZZZ):H\in\Rddsym\}$. Given an observation $Y$ in $\GGG$ the property $\ZZZ\vc(H)=2\vc(H)$ implies that the best unbiased estimator of $\vc(H)$ is given by $\frac{1}{2}\III(\Sigma)^{-1}Y\sim\NNN(\vc(H),\frac{1}{4}\III(\Sigma)^{-1}\ZZZ)$. This determines the asymptotic distribution of regular estimators, which is made precise in the following.

\subsection{\textbf{Verification of Theorem~\ref{ThmConvPara}}}
\begin{proof}
If one closely follows the steps as in the verification of the general (convolution) Theorem 3.11.2 in \citet{VanDVaart[2013]} then the only peculiarity to be taken into account is the matrix $\ZZZ$. More precisely, for an orthonormal basis $h_1,\ldots,h_{d^*},\ d^*:=d(d+1)/2$, of $\vc(\Rddsym):=\{\vc(A):A\in\Rddsym\}$ with respect to the inner product $\langle \cdot,\cdot\rangle_{\III(\Sigma)\ZZZ}$, Proposition~\ref{LAN_gen} and Le Cam's Third Lemma yield
\begin{equation}\label{LimitDecomp}
r^{-1}_n(\hat\vartheta_n-\psi(\Sigma+r_nH))\overset{d}{\to}\NNN\Big(0,\sum^{d^*}_{k=1}\nabla\psi_{\Sigma}h_kh^{\top}_k\nabla\psi^{\top}_{\Sigma}\Big)\ast R,
\end{equation}
under $P^n_{\Sigma+r_nH}$, for some $R$. The independence of $h$ now follows by
\begin{align*}
&\Big(\sum^{d^*}_{k=1}\nabla\psi_{\Sigma}h_kh^{\top}_k\nabla\psi^{\top}_{\Sigma}\Big)_{i,j}=\sum^{d^*}_{k=1}\langle\nabla\psi_{\Sigma}^{(i)},h_k\rangle\langle\nabla\psi_{\Sigma}^{(j)},h_k\rangle\\
=&\frac{1}{4}\sum^{d^*}_{k=1}\langle\III^{-1}_{\Sigma}\nabla\psi_{\Sigma}^{(i)},h_k\rangle_{\III(\Sigma)\ZZZ}\langle\III^{-1}_{\Sigma}\nabla\psi_{\Sigma}^{(j)},h_k\rangle_{\III(\Sigma)\ZZZ}=\frac{1}{4}\langle\nabla\psi^{(i)}_{\Sigma},\III^{-1}_{\Sigma}\nabla\psi^{(j)}_{\Sigma}\rangle_{\ZZZ},
\end{align*}
where $\nabla\psi^{(i)}_{\Sigma}$ denotes the $i$-th column of $\nabla\psi^{\top}_{\Sigma}$ and $1\leq i,j\leq d$.
\end{proof}
\begin{remark}\label{RmkZZZ}
Note that the singularity of $\III(\Sigma)\ZZZ$ has no critical impact as $\langle\cdot,h_k\rangle_{\III(\Sigma)\ZZZ}=2\langle\cdot,h_k\rangle_{\III(\Sigma)}$ is the essential isometry-type ingredient used.
\end{remark}
\subsection{\textbf{Estimation}}\label{estimation}
For each observation $Y_p$ in \eqref{SeqPar} an unbiased estimator of $\psi(\Sigma)=\vc(\Sigma)$ can be obtained via
\[\hat\vartheta_p:=\lambda_p^{-1}\vc\Big(Y_{np}Y^{\top}_{np}-\frac{\eta^2}{n}I_d\Big).\]
Since $\hat\vartheta_p,\ p\geq1,$ are independent it is reasonable to consider a weighted average to reduce variability. Let $\pi_n=[\underline \pi_n,\overline\pi_n]\cap\N$ be as in Theorem~\ref{ThmRateFisher} and set $\III_J(\Sigma):=\sum_{p\in J}\IInp(\Sigma)$, for $J\subseteq\N$. Then, by a Lagrange approach, the choice of weights
\[W_p(\Sigma):=\IItr(\Sigma)^{-1}\IInp(\Sigma)\]
ensures unbiasedness and minimal covariance of the oracle estimator
\[\hat\vartheta^{\text{or}}_n:=\sum_{p\in\pi_n}W_p(\Sigma)\hat\vartheta_p.\]
Let $\pi'_n\subsetneq\N$ be with $\pi_n\cap\pi'_n=\emptyset$, $n\geq1$, and $|\pi'_n|\to\infty$, as $n\to\infty$. Set $\hat\vartheta^{\text{pre}}_n:=\sum_{p\in\pi'_n}W^{\pi'_n}_p(SI_d)\hat\vartheta_p$, where $W^{\pi'_n}_p(\Sigma):=\III_{\pi'_n}(\Sigma)^{-1}\IInp(\Sigma)$, and set $\hat\Sigma^{\text{pre}}_n:=\text{mat}(\hat\vartheta^{\text{pre}}_n)$, where $\text{mat}:\R^{d^2}\to\R^{d\times d}$ is the inverse of $\vc$. Then an adaptive version of $\hat\vartheta^{\text{or}}_n$ is obtained by
\begin{equation}\label{adaptest}
\hat\vartheta^{\text{ad}}_n:=\sum_{p\in\pi_n}W_p(\hat\Sigma^{\text{pre}}_n)\hat\vartheta_p.
\end{equation}
Note that it is crucial that $(W_p(\hat\Sigma^{\text{pre}}_n))_{p\in\pi'_n}$ is independent of $(Y_p)_{p\in\pi_n}$.
\begin{theorem}\label{ThmOrAd}
The estimators $\hat\vartheta^{\text{or}}_n$ and $\hat\vartheta^{\text{ad}}_n$ of $\psi(\Sigma)=\vc(\Sigma)$ are regular and efficient in the sense of Theorem~\ref{ThmConvPara}. In particular, it holds that
\[r^{-1}_n(\hat\vartheta^{ad}_n-\psi(\Sigma+r_nH))\overset{d}{\to}\NNN(0,\tfrac{1}{4}\III(\Sigma)^{-1}\ZZZ),\quad\text{as }n\to\infty,\]
under $P^n_{\Sigma+r_nH}$, for any $H\in\R^{d\times d}_{\text{sym}}$.
\end{theorem}
\begin{remark}
The estimator $\hat\vartheta^{\text{ad}}_n=\hat\vartheta^{\text{ad}}_n((Y_p)_{p\geq1})$ in $\ExpFS$ can be obtained in the initial model $\ExpF$ by the explicit construction via interpolations given in the proof of Theorem~\ref{ThmLeCamDiscCont}. In particular, for an interpolated version $(\bar Y_t)_{t\in[0,1]}$ of \eqref{DiscGen}, cf. \eqref{contmodel_zwischen}, the estimator $\hat\vartheta^{\text{ad}}_n=\hat\vartheta^{\text{ad}}_n((\bar Y_p)_{p\geq1})$ in $\ExpF$ can be built as in \eqref{adaptest} from
\[\bar Y_p:=((e_{pj},\bar Y))_{1\leq j\leq d},\quad p\geq1,\]
where $e_{pi}=(\1\{i=j\}\varphi_p)_{1\leq j\leq d}$ and $\varphi_p$ is the eigenfunction corresponding to $\lambda_p$, cf. Section~\ref{secCtsSeq}. For the limit distribution of $\hat\vartheta^{\text{ad}}_n((\bar Y_p)_{p\geq1})$ note that for $\bar P^n_{\Sigma}:=\LLL((\bar Y_p)_{p\geq1})$ and $f$ continuous and bounded it easily can be seen that
\[\E_{\bar P^n_{\Sigma}}[f(\hat\vartheta^{\text{ad}}_n)]=\E_{P^n_{\Sigma}}[f(\hat\vartheta^{\text{ad}}_n)]+\OOO(\|f\|_{\infty}\|P^n_{\Sigma}-\bar P^n_{\Sigma}\|_{\text{TV}})\]
where the total variation norm satisfies $\|P^n_{\Sigma}-\bar P^n_{\Sigma}\|_{\text{TV}}\to0$, by the proof of Theorem~\ref{ThmLeCamDiscCont}. In particular, the estimator $\hat\vartheta^{\text{ad}}_n((\bar Y_p)_{p\geq1})$ has the same asymptotic properties as its counterpart constructed in $\ExpFS$ and it satisfies the statement of Theorem~\ref{ThmOrAd}.
\end{remark}
\subsection{\textbf{Further asymptotic equivalences}}\label{morecam}
The adaptive estimator $\hat\vartheta^{\text{ad}}_n$ in \eqref{adaptest} allows for further asymptotic equivalence statements that completes the asymptotic analysis of the fundamental parametric model $\ExpF$. By Theorem~\ref{ThmRateFisher} the asymptotically significant information for estimating $\Sigma$ efficiently in $\ExpFS$ is already contained in the subexperiment $\FFF^s_{n,\pi_n}$ that is generated by the observations $(Y_p)_{p\in\pi_n}$, where $\pi_n$ is as in Theorem~\ref{ThmRateFisher}, i.e.,
\[\pi_n=[a_np_n,b_np_n]\cap\N,\quad a_n\downarrow0,\quad b_n\to\infty.\]
Clearly, $\ExpFS$ is at least as informative as $\FFF^s_{n,\pi_n}$, but even the reverse can be shown, at least asymptotically, given that the parameter set $\Spar$ is replaced by the more restrictive set (with $S>1$)
\begin{equation}\label{sparprime}
\Spar'=\{\Sigma\in\Rddsym:S^{-1}I_d<\Sigma<SI_d\}.
\end{equation}
\begin{proposition}\label{PropLeCamProj}
For parameter set $\Spar'$ in \eqref{sparprime} the experiments $\ExpFS$ and $\FFF^s_{n,\pi_n}$ are asymptotically equivalent in Le Cam's sense. More precisely,
\[\Delta(\ExpFS,\FFF^s_{n,\pi_n})=\OOO(S\log n)\OOO(a^{\delta-1/2}_n\vee b^{1/2-\delta}_n).\]
\end{proposition}
Proposition~\ref{PropLeCamProj} gives a further intuition on smoothing choices for several known estimation methods such as pre-averaging, where the frequencies of order $\sqrt{n}$ play a central role for models driven by a Brownian motion, cf. \citet{Jacod[2009]}.

Next the impact of deviations in the underlying eigenvalue sequence $\lambda$ is investigated. As we have seen in Theorem~\ref{ThmRateFisher}, the leading term of $(\lambda_p)_{p\geq1}$ completely determines the asymptotic lower bounds. As an example consider the cases in which $G$ in \eqref{ContPar} is a Brownian bridge or a Brownian motion. The respective underlying eigenvalue sequences read as
\[\lambda^{\text{BB}}_p=(\pi p)^{-2}\quad\text{and}\quad\lambda^{\text{BM}}_p=\pi^{-2}(p-1/2)^{-2},\]
respectively, and thus the bounds obtained by Theorem~\ref{ThmConvPara} coincide. In fact, even a general characterisation of asymptotic equivalence on the basis of the underlying eigenvalue sequence can be given.

\begin{proposition}\label{PropEigLeCamEqui}
For sequences $\lambda$ and $\lambda'$ satisfying Assumption~\ref{regvarass}-$\lambda(\delta)$ (with possibly different $\delta$) let $\ExpFS$ and $\FFF^{s'}_n$, respectively, be sequence space models of type \eqref{SeqPar} on $\Spar'$ as in \eqref{sparprime}. Then the following are equivalent:
\begin{enumerate}
\item $\lambda_p/\lambda'_p\to1$, as $p\to\infty$.
\item $r_n/r'_n\to1$, as $n\to\infty$, and $\III(\Sigma)=\III'(\Sigma)$, for all $\Sigma\in\Spar$.
\item $\Delta(\ExpFS,\FFF^{s'}_n)\to0$, as $n\to\infty$,
\end{enumerate}
where $r'_n$ and $\III'(\Sigma)\ZZZ$ are the rate and asymptotic Fisher information in $\FFF^{s'}_n$.
\end{proposition}
\newpage
The impact of the leading term of $\lambda$ yields an interesting finding in the particular scenario, in which the signal process is a mixture
\[G_t=Z_{1,t}+Z_{2,t},\]
of two independent Gaussian processes $Z_i=(Z_{i,t})_{t\in[0,1]},\ i=1,2$. If the covariance operators of $Z_1$ and $Z_2$ are diagonalisable by the same basis then the process with more slowly decaying eigenvalues completely determines the asymptotic properties of the estimation problem. Therefore one might conjecture for $G$ being a so-called mixed fractional Brownian motion of Hurst index $H>1/2$, cf. \citet{Cheridito[2001]}, that solely the Brownian motion part contributes to the underlying asymptotics.


\section{Semiparametric efficiency under asynchronicity}\label{SecSemiPara}
In the following we suppose that Assumptions~\ref{AssSigma}-$\Sigma(\beta,M,S)$ and \ref{AssRegF}-$F(\gamma,N,\beta)$ hold.
\subsection{\textbf{Locally parametric approximation}}\label{SsecPwConAppr}
As in the parametric set-up, observing \eqref{DiscSemi} is approximated by its continuous analogue. However, in order to use the parametric results, locally constant approximations of $\Sigma$ and $F=(F_j)_{1\leq j\leq d}$ are considered. More precisely, for $m$ disjoint blocks $\I_{mk}:=[k/m,(k+1)/m),k=0,\ldots,m-1$, and $\Sigma_{m,k}:=\Sigma(k/m)$ introduce
\[\Sigma_m:=\sum^{m-1}_{k=0}\Sigma_{m,k}\1_{\I_{mk}}(\cdot),\quad F'_{j,m}:=\sum^{m-1}_{k=0}F'_j\Big(\frac{k}{m}\Big)\1_{\I_{mk}}(\cdot),\quad j=1,\ldots,d.\]
and the corresponding continuous observation model
\begin{equation}\label{ContSemi}
dY^m_t=\Big(\int^t_0\Sigma^{1/2}_m(s)dB_s\Big)dt+\Xi_m(t)dW_t,\quad t\in[0,1],
\end{equation}
where
\[\Xi^2_m:=\diag(\eta^2_j/(n_jF'_{j,m}))_{1\leq j\leq d}.\]
\begin{definition}\label{DefExpMCm}
For $n:=(n_1,\ldots,n_d)$ let $\ExpM$ and $\ExpMC$ be the statistical experiments that are generated by the discrete and continuous observations \eqref{DiscSemi} and \eqref{ContSemi}, respectively.
\end{definition}
Let $\lambda_{mp}:=(\pi pm)^{-2}$, $p\geq1$, and set $n_{\max}:=\max_{1\leq j\leq d}n_j$. The Le Cam distance between $\ExpM$ and $\ExpMC$ is bounded by the approximation errors of $\Sigma_m$ and $F'_{j,m}$. As $m$ will have to be chosen later in this section such that $m=o(\sqrt{n_{\min}})$, the restriction $\beta>1/2$ is evident in view of the following.
\begin{proposition}\label{PropPwc}
For any $\kappa\in(0,1/2)$ and $n_{\min}\to\infty$ it holds that
\[\Delta(\ExpM,\ExpMC)=\OOO(MSn_{\max}n^{-3/2+\kappa}_{\min})+\OOO(MSn^{1/4}_{\max}m^{-\beta}).\]
In particular, asymptotic equivalence holds, given that $m=o(\sqrt{n_{\min}})$.
\end{proposition}
\subsection{\textbf{LAN for correlated and uncorrelated sequence space models}}
As described in Section~\ref{secCtsSeq} a continuous experiment can be represented in the sequence space. To this end, consider the (normalised) $L^2([0,1],\R)$-basis
\begin{align*}
\varphi_{0,0}(t)&:=\sqrt{m}\1_{\I_{m,0}}(t),\\
\varphi_{0,k+1}(t)&:=\frac{\sqrt{m}}{\sqrt{2}}(\1_{\I_{mk}}(t)-\1_{\I_{m,k+1}}(t)),\quad k=0,\ldots,m-2,\\
\varphi_{pk}(t)&:=\sqrt{2m}\cos(p\pi(tm-k))\1_{\I_{mk}}(t),\quad p\geq1,\quad k=0,\ldots,m-1.
\end{align*}
Via $e_{pki}=(\1_{\{i=j\}}\varphi_{pk})_{1\leq j\leq d}$ Gaussian random vectors
\begin{equation}\label{SeqSemiTwo}
S_{pk}:=((e_{pki},dY^m))_{1\leq i\leq d},\quad p\geq0,\ k=0,\ldots,m-1
\end{equation}
are obtained, cf. \eqref{cylmeas}. Clearly observing the correlated vectors $(S_{pk})_{p\geq 0,k=0,\ldots,m-1}$ is equivalent to observing \eqref{ContSemi} and more informative than observing $(S_{pk})_{p\geq 1,k=0,\ldots,m-1}$. However, the latter sequence is independent and close to observing \eqref{SeqSemi}, hence it is similar to experiment $\ExpFS$ which has been intensively studied in Section~\ref{SecParametric}.
\begin{proposition}\label{PropLANequi}
Let $m=o(\sqrt{n_{\min}})$ be satisfied. Then any LAN-expansion with respect to $\Sigma+n^{-1/4}_{\min}H,\ H\in\HBsym$ for the model \eqref{SeqSemi} is also valid in $\ExpMC$ and $\ExpM$.
\end{proposition}

\subsection{\textbf{Verification of Theorem~\ref{ThmConvNonpara}}}
\begin{proof}
The score induced by \eqref{SeqSemi} equals $\nabla\ell_{n}(\Sigma):=\vc(\nabla\ell^{(0)}_n(\Sigma),\ldots,\nabla\ell^{(m-1)}_n(\Sigma))$, where $\nabla\ell^{(k)}_n(\Sigma)$ is of the exact same shape as the parametric score in \eqref{SeqScore} with $\lambda_{mp},\ C_{pk}$ and $Y_{pk}$ replacing $\lambda_p,\ C_p$ and $Y_p$, respectively. Therefore the (not $\ZZZ$-normalised) Fisher information in $\ExpMS$ is given by the block diagonal matrix
\[\III_{n,m}(\Sigma):=\begin{pmatrix}\IIn^{(0)}(\Sigma)&0&\cdots&0\\0&\IIn^{(1)}(\Sigma)&\cdots&0\\\vdots&&\ddots&\vdots\\0&\cdots&\cdots&\IIn^{(m-1)}(\Sigma)\end{pmatrix},\]
with blocks
\[\IIn^{(k)}(\Sigma):=\frac{1}{4}\sum^{\infty}_{p=1}\lambda^2_{mp}(C^{-1}_{pk}\otimes C^{-1}_{pk}),\ k=0,\ldots,m-1.\]
As in Theorem~\ref{ThmRateFisher}, regular variation of the eigenvalues $\lambda$ yields that on each block $\I_{mk}$ the Fisher information grows with rate $\sqrt{n_{\min}}/m$ such that
\begin{equation}\label{FisherRiemann}
n^{-1/2}_{\min}\sum^{m-1}_{k=0}\IIn^{(k)}(\Sigma)=\frac{1}{m}\sum^{m-1}_{k=0}\III_{\Sigma}(k/m)+o(1)\to\int^1_0\III_{\Sigma}(t)dt,
\end{equation}
i.e., the rate is $n^{-1/4}_{\min}$, where (cf. proof of Theorem~\ref{ThmRateFisher} and Remark~\ref{RmkFisherExp})
\[\III_{\Sigma}(t)=\frac{1}{8}(\Sigma^{1/2}_{\Xi}(t)\otimes\Sigma(t)+\Sigma(t)\otimes \Sigma^{1/2}_{\Xi}(t))^{-1},\ t\in[0,1].\]
For $H\in\HBsym$ - (as before) in the sense that $\Sigma+n^{-1/4}_{\min}H\in\Snon$, for $n$ sufficiently large - note that \eqref{LANrem1} and \eqref{LANrem2} hold uniformly in $H_{m,k}:=H(k/m),\ k=0,\ldots,m-1$. Thus applying Proposition~\ref{LAN_gen} simultaneously leads to (denoting by $Q^n_{\Sigma}$ the measure induced by \eqref{SeqSemi})
\begin{align*}
\log\frac{dQ^n_{\Sigma+n^{-1/4}_{\min}H}}{dQ^n_{\Sigma}}=\sum^{m-1}_{k=0}\Big(&\vc(H_{m,k})^{\top}\nabla\ell^{(k)}_n(\Sigma)-\frac{1}{2\sqrt{n_{\min}}}\|H_{m,k}\|^2_{\IIn^{(k)}(\Sigma)\ZZZ,L^2}\\
&+\rho^{(1,k)}_n+\rho^{(2,k)}_n\Big),
\end{align*}
where \eqref{FisherRiemann} implies $n^{-1/2}_{\min}\sum^{m-1}_{k=0}\|H_{m,k}\|^2_{\IIn^{(k)}(\Sigma)\ZZZ,L^2}\to\|H\|^2_{\III_{\Sigma}\ZZZ,L^2}$ with $\langle H,H\rangle_{\III_{\Sigma}\ZZZ,L^2}:=\int^1_0\langle H(t),H(t)\rangle_{\III_{\Sigma}(t)\ZZZ}dt$ (similarly for $\|\cdot\|_{\IIn^{(k)}(\Sigma)\ZZZ,L^2}$). Moreover, for $k=0,\ldots,m-1$, \eqref{LANrem1} and \eqref{LANrem2} imply $\E[\rho^{(1,k)}_n]=0$ as well as
\[\var\Big(\sum^{m-1}_{k=0}\rho^{(1,k)}_n\Big)=\OOO(n^{-1/2}_{\min}),\quad\sum^{m-1}_{k=0}|\rho^{(2)}_{n,k}|=\OOO(n^{-1/4}_{\min}).\]
Since a central limit theorem applies for $n^{-1/4}_{\min}\sum^m_{k=1}\vc(H_k)\nabla\ell^{(k)}_n(\Sigma)$ analogously as in Theorem~\ref{LAN} the sequence of experiments $\ExpMS$ satisfies
\begin{equation}\label{LANNonpara}
\log\frac{dP^{n,m}_{\Sigma+n^{-1/4}_{\min}H}}{dP^{n,m}_{\Sigma}}=\Delta_{n,\Sigma,H}+\frac{1}{2}\|H\|^2_{\III_{\Sigma}\ZZZ,L^2},\quad H\in\HBsym,
\end{equation}
where $\Delta_{n,\Sigma,H}\overset{d}{\to}\Delta_{\Sigma,H}$, under $Q^n_{\Sigma}$, with $\Delta_{\Sigma,H}$ being the centred Gaussian process with $\cov(\Delta_{\Sigma,H_1},\Delta_{\Sigma,H_2})=\langle H_1,H_2\rangle_{\III_{\Sigma}\ZZZ,L^2}$.

In order to establish a convolution theorem, the verification of Theorem 3.11.2 in \citet{VanDVaart[2013]} is once more closely followed. First denote for the asymptotic perturbation error by
\[\dot\kappa(H):=\int^1_0(\nabla W_{\Sigma}\vc(H))(t)dt=\lim_{n_{\min}\to\infty}n^{1/4}_{\min}(\psi(\Sigma+n^{-1/4}_{\min}H)-\psi(\Sigma)).\]
For $U\geq1$ let $L_U$ be a $U$-dimensional subspace of $\HBsym$ and let $H_1,\ldots,H_U$ be an orthonormal basis of $L_U$ with respect to $\langle\cdot,\cdot\rangle_{\III_{\Sigma}\ZZZ,L^2}$. Denote by $\dot W^{(i)}_{\Sigma}$ the $i$-th column of $(\nabla W_{\Sigma})^{\top}$ and let $h_u:=\vc(H_u),\ u=1\ldots,U$. Then \eqref{LANNonpara} and Le Cam's third Lemma yield that the limit distribution of regular estimators under $Q^n_{\Sigma+n^{-1/4}_{\min}H}$, $H\in L_U$, is a convolution of some $R$ with $\NNN(0,\sum^U_{u=1}\dot\kappa(H_u)\dot\kappa(H_u)^{\top})$, cf. \eqref{LimitDecomp}. Thus the $(i,j)$-entry of the optimal asymptotic covariance of estimating $\psi(\Sigma+n^{-1/4}_{\min}H)$ is obtained by a limiting argument and (once more) by the properties of $\ZZZ$ via
\begin{align*}
&\lim_{U\to\infty}\sum^U_{u=1}(\dot\kappa(H_u)\dot\kappa(H_u)^{\top})_{i,j}=\lim_{U\to\infty}\sum^U_{u=1}\langle\dot W_{\Sigma}^{(i)},h_u\rangle_{L^2}\langle\dot W_{\Sigma}^{(j)},h_u\rangle_{L^2}\\
=&\lim_{U\to\infty}\frac{1}{4}\sum^U_{u=1}\langle\III^{-1}_{\Sigma}\dot W_{\Sigma}^{(i)},h_u\rangle_{\III_{\Sigma}\ZZZ,L^2}\langle\III^{-1}_{\Sigma}\dot W_{\Sigma}^{(j)},h_u\rangle_{\III_{\Sigma}\ZZZ,L^2}\\
=&\frac{1}{4}\langle\III^{-1}_{\Sigma}\dot W_{\Sigma}^{(i)},\III^{-1}_{\Sigma}\dot W_{\Sigma}^{(j)}\rangle_{\III_{\Sigma}\ZZZ,L^2}.
\end{align*}
\end{proof}

\appendix
\section{Le Cam equivalence, LAN and regular variation}\label{Appendix}
\subsection{The Le Cam $\Delta$-distance}\label{SsecLeCamDistance}
Next some facts of Le Cam theory are given, cf. \cite{LeCam[2000]} and \cite{Mariucci[2016]} for an overview. For some set $\Theta$ of parameters let $\EEE=\{(Y,\mathcal Y,P_{\theta}):\theta\in\Theta\}$ and $\FFF=\{(Y,\mathcal Y, Q_{\theta}):\theta\in\Theta\}$ be two statistical experiments on a common Polish space $(Y,\mathcal Y)$. Then it holds that
\begin{equation}\label{lecamhell}
\Delta(\EEE,\FFF)\leq\sup_{\theta\in\Theta}\|P_{\theta}-Q_{\theta}\|_{\text{TV}}\leq\sup_{\theta\in\Theta}H(P_{\theta},Q_{\theta}).
\end{equation}
Here $H(P,Q):=(\int(\sqrt{f}-\sqrt{g})^2d\mu)^{1/2}$ denotes the Hellinger distance, where $P$ and $Q$ are probability measures with $\mu$-densities $f$ and $g$, respectively. For Gaussian laws $P\sim\NNN(\mu_1,\Sigma_1)$ and $Q\sim\NNN(\mu_2,\Sigma_2)$ on $\R^d$ with invertible covariance matrices $\Sigma_1,\Sigma_2\in\R^{d\times d}$ it is well-known (cf. \citet{Reiss[2011]}) that
\begin{equation}\label{hellingergauss}
H^2(P,Q)\leq 4\|\Sigma^{-1/2}_1(\mu_1-\mu_2)\|^2+\frac{1}{2}\|\Sigma^{-1/2}_1(\Sigma_2-\Sigma_1)\Sigma^{-1/2}_1\|^2.
\end{equation}
More generally, let $\GGG_i(\Theta)=\{(X,\XXX,\NNN_{\mu_i,K_i}:\theta\in\Theta\},\ i=1,2$, where $\NNN_{\mu_i,K_i}$ is a (possibly cylindrical) Gaussian measure on some Hilbert space $X$, such that both, the mean $\mu_i\in X$ and the positive self-adjoint covariance operator $K_i:X\to X$, are driven by $\theta$. Combining \eqref{lecamhell} with the infinite-dimensional analogue of \eqref{hellingergauss} yields
\begin{equation}\label{lecamhilbert}
\Delta(\GGG_1,\GGG_2)\lesssim\sup_{\theta\in\Theta}\|K^{-1/2}_1(\mu_1-\mu_2)\|_{L^2}\vee\|K^{-1/2}_1(K_2-K_1)K^{-1/2}_1\|_{\text{HS}},
\end{equation}
where $\|\cdot\|_{\text{HS}}$ denotes the Hilbert-Schmidt norm on $X$. Note that for integral operators $K$ with kernel $k$ one has
\begin{equation}\label{HSL2}
\|K\|_{\text{HS}}=\|k\|_{L^2}.
\end{equation}
\subsection{Weak convergence, LAN and regular estimators}\label{SsecWeakLanReg}
Let $\Theta$ be an open subset of a linear subspace $\HHH$ of some Hilbert space. A sequence of experiments $\EEE_n=\{P^n_{\theta}:\theta\in\Theta\}$ on Polish spaces is said to converge weakly to an experiment $\EEE=\{P_{\theta}:\theta\in\Theta\}$ if
\[\Delta(\EEE_n(I),\EEE(I))\to0,\quad \text{as }n\to\infty,\]
for any finite $I\subseteq\Theta$. Assume that $P_{\theta}\ll P_{\theta'}$ and $P^n_{\theta}\ll P^n_{\theta'}$, for any $\theta,\theta'\in\Theta$ and $n\in\N$. Then weak convergence of $\EEE_n$ to $\EEE$ is equivalent to
\[\LLL\Big(\Big(\frac{dP^n_{\theta'}}{dP^n_{\theta}}\Big)_{\theta'\in I}|P^n_{\theta}\Big)\to\LLL\Big(\Big(\frac{dP_{\theta'}}{dP_{\theta}}\Big)_{\theta'\in I}|P_{\theta}\Big),\]
for any finite $I\subseteq\Theta$, for any $\theta\in\Theta$. This means that verification of the LAN-property for $\EEE_n$ with rate $r_n$ implies weak convergence of $r_n$-localisations of $\EEE_n$ to a normal limit experiment. Since the distance $\Delta$ satisfies the triangle inequality, the LAN-property of the sequence $\EEE_n$ carries over to sequences of experiments $\FFF_n$ whose $r_n$-localisation are asymptotically equivalent to the one of $\EEE_n$ (at least for finite parameter subsets).

A sequence of estimators $\hat\vartheta_n$ of a target $\psi(\theta)\in\R^k$ is called regular if
\[r^{-1}_n(\hat\vartheta_n-\psi(\vartheta+r_nh))\overset{d}{\to}L_{\vartheta},\]
under $P^n_{\theta+r_nh}$, with limit distribution $L_{\vartheta}$ that does not depend on $h\in\HHH$.
\subsection{Regular variation}
In the following let $f,g:(0,\infty)\to(0,\infty)$ be regularly varying. Then an immediate consequence is the following.\newpage
\begin{proposition}\label{PropRegVar}
If $f$ and $g$ are regularly varying with index $\delta$ then
\begin{enumerate}
\item $f^2$ is regularly varying with index $2\delta$,
\item $1/f$ is regularly varying with index $-\delta$,
\item $f/g$ is slowly varying.
\end{enumerate}
\end{proposition}
An important property of regularly varying functions is the following uniformity result that is stated as Theorem 1.5.2 in \citet{Bingham[1989]}.
\begin{theorem}\label{ThmRegVarUni}
For a regularly varying function $f$ the convergence
\[\lim_{x\to\infty}\frac{f(ax)}{f(x)}=a^{\delta},\]
holds uniformly
\begin{enumerate}
\item on each $[y,z]$, if $\delta=0$ (i.e., if $f$ is slowly varying),
\item on each $[y,\infty)$, if $\delta<0$,
\item on each $(0,y]$, if $\delta>0$ and if $\sup_{x\in(0,y]}f(x)<\infty$, for any $y>0$.
\end{enumerate}
\end{theorem}

\section{Asymptotic equivalence between discrete and continuous Gaussian models}
\label{SecAsyEqui}
\subsection{General Gaussian models}
In the following, discrete and continuous versions of a universal Gaussian model are introduced that are kept as general as possible in the sense that the unknown parameter consists of the mean and covariance function itself.

For $n=(n_1,\ldots,n_d)\in\N^d$ consider the discrete observation model
\begin{equation}\label{DiscGen}
Y_{i,j}=(\mu(t_{i,j})+Z_{t_{i,j}})_j+\varepsilon_{i,j},\quad 1\leq i\leq n_j,\ 1\leq j\leq d,
\end{equation}
where $Z=(Z_t)_{t\in[0,1]}$ denotes a centred $d$-dimensional Gaussian process which is independent of the mutually independent noise variables $\varepsilon_{i,j}\sim\NNN(0,\eta^2_j),\ 1\leq i\leq n_j$, with $\eta_j>0$ known, for $1\leq j\leq d$. It is assumed that under $n_{\min}\to\infty$ one has $n_j/n_{\min}\to\nu_j$ for some $\nu_j\in(0,1],\ 1\leq j\leq d$. Moreover, $t_{i,j}$ relates to some distribution function $(F_j)_{1\leq j\leq d}$ as in Assumption~\ref{AssRegF}-$F(\gamma,N,\beta)$. The parameter of interest is given by $(\mu,k)\in\Theta:=H^{\alpha}_L\times H^{\beta'}_M$, for some $\alpha\in(1/2,2),\ \beta'\in(1,2)$ and $L,M>0$, where $k(s,t):=\cov(Z_s,Z_t),\ s,t\in[0,1]$. Let $n_{\max}=\max_{1\leq j\leq d}n_j$ and let further
\begin{equation}\label{ContGen}
dY_t=(\mu(t)+Z_t)dt+\Psi_n(t)dW_t,\quad t\in[0,1],
\end{equation}
where $\Psi^2_n(t):=\diag((\eta^2_j/n_jF'_j(t))_{1\leq j\leq d})$.
\begin{definition}
Denote by $\ExpG$ and $\ExpGC$ the experiments that are generated by the observations in \eqref{DiscGen} and \eqref{ContGen}, respectively.
\end{definition}
\begin{remark}
It is evident that $\ExpF$ and $\ExpM$ are just special cases of $\ExpG$ and that $\ExpGC$ nests the models \eqref{ContPar} and $\ExpMCTwo$, where the latter is an approximation of $\ExpMC$ that is used in the proof of Proposition~\ref{PropPwc}.
\end{remark}

\subsection{Asymptotic equivalence}\label{SsecAsyEquiSobolev}
\begin{theorem}\label{ThmLeCamDiscCont}
Let Assumption~\ref{AssRegF}-$F(\gamma,N,\beta'-1)$ be satisfied with $\gamma>\alpha-1$.
Then $\GGG_n$ and $\GGG^c_n$ are asymptotically equivalent. In particular,
\[\Delta(\GGG_n,\GGG^c_n)=\OOO(L\eta^{-1}n^{1/2}_{\max}n^{-\alpha}_{\min}\vee M\eta^{-2}n_{\max}n^{-\beta'}_{\min}).\]
\end{theorem}

The above asymptotic equivalence result holds uniformly over a large class of Gaussian processes. Note that $\alpha>1/2$ and $\beta'>1$ are common sufficient (and often necessary) assumptions among uni- and bi-variate asymptotic equivalence results, cf. \citet{Reiss[2008]}. In order to gain from higher regularities $\alpha,\beta'\geq2$ more derivatives have to be controlled, e.g. by not only piecewise constant approximations, but this lies beyond the scope of this work.
\begin{proof}[\textbf{Proof of Theorem~\ref{ThmLeCamDiscCont}}]
For $\I_{n_j,i}:=((i-1)/n_j,i/n_j]$ and $g_{ij}(t):=\1_{\I_{n_j,i}}(F_j(t)),\ 1\leq i\leq n_j,\ 1\leq j\leq d,$ consider the continuous observation
\begin{equation}\label{contmodel_vor}
\tilde Y_t:=\Big(\sum^{n_j}_{i=1}Y_{i,j} g_{ij}(t)\Big)_{j=1,\ldots,d},\quad t\in[0,1].
\end{equation}
Note that observing \eqref{contmodel_vor} is equivalent to the observations in \eqref{DiscGen} and that
\[\cov(\tilde Y_s,\tilde Y_t)=\cov(\Pi_nZ_s,\Pi_nZ_t)+\diag\Big(\Big(\eta^2_j\sum^{n_j}_{i=1}g_{ij}(s)g_{ij}(t)\Big)_{1\leq j\leq d}\Big),\]
where $\Pi_n Z_t=(\sum^{n_j}_{i=1}(Z_{t_{i,j}})_j\1_{\I_{n_j,i}}(F_j(t)))_{1\leq j\leq d}$. The Cauchy-Schwarz inequality implies for $f=(f_1,\ldots,f_d)\in L^2([0,1],\R^d)$
\[\sum^{n_j}_{i=1}\langle f_j,g_{ij}\rangle^2_{L^2}\leq \frac{1}{n_j}\sum^{n_j}_{i=1}\int_{\I_{n_j,i}}\frac{f_j(F^{-1}_j(t))^2}{F'_j(F^{-1}_j(t))^2}dt=\frac{1}{n_j}\int^1_0\frac{f_j(t)^2}{F'_j(t)}dt\]
Thus by adding uninformative noise the observation
\begin{equation}\label{contmodel_zwischen}
d\bar Y_t:=\Pi_n(\mu(t)+Z_t)dt+\Psi_n(t)dW_t,\quad t\in[0,1],
\end{equation}
can be constructed from \eqref{contmodel_vor}. On the other hand, it is easy to see that the law of $(Y_{i,j})_{i=1,\ldots,n_j;j=1\ldots,d}$ coincides with
\[\LLL\Big(\Big(n_j\int_{F^{-1}_j(I_{n_j,i})}F'_j(t)d(\bar Y_t)_j\Big)_{i=1,\ldots,n_j;j=1\ldots,d}\Big),\]
i.e., observations of type \eqref{contmodel_vor} can be constructed from \eqref{contmodel_zwischen}. In particular, \eqref{DiscGen} and the experiments generated by \eqref{contmodel_vor} and \eqref{contmodel_zwischen} are equivalent.

Next it is shown that \eqref{ContGen} and the experiment generated by \eqref{contmodel_zwischen} are asymptotically equivalent. Denote by $K$ and $K_{\Psi_n}$ the covariance operators of $Z$ and $\Psi_ndW$. Then with $n_{\max}:=\max_{1\leq j\leq d}n_j$ and $c^2_{\eta}:=\max_{1\leq j\leq d}\|F'_j\|_{\infty}/\eta^2_j$ the bound $(K+K_{\Psi_n})^{-1/2}\leq\sqrt{n_{\max}}c_{\eta}\Id$ along with \eqref{lecamhilbert} and \eqref{HSL2} gives $\Delta(\GGG_n,\GGG^c_n)=\OOO(\psi_n(\Theta))$ with
\[\psi_n(\Theta):=\sup_{\theta\in\Theta}\Big(\sqrt{n_{\max}}c_{\eta}\|\mu-\Pi_n\mu\|_{L^2(\R^d)}\vee n_{\max}c^2_{\eta}\|k-k_{\Pi_n}\|_{L^2(\R^{d\times d})}\Big),\]
and $k_{\Pi_n}$ being the covariance function of $\Pi_nZ$. In particular it suffices to show $\psi_n(\Theta)=o(1)$ to obtain asymptotic equivalence. For this note that for fixed $\alpha\in(1/2,2)$ all $f\in H^{\alpha}(\Omega,\R)$ vanishing at some $x_0\in\Omega$ (with $\Omega=[0,1]$) obey the uniform bound
\begin{equation}\label{Sobolev_bound}
\|f\|_{L^2(\Omega)}\leq c_{\alpha}|f|_{H^{\alpha}(\Omega)},
\end{equation}
with $c_{\alpha}>0$. The bound \eqref{Sobolev_bound} can be obtained by contradiction in a similar way as the Poincar\'e inequality, cf. Chapter 5.8.1 in \citet{Evans[2010]}. By a scaling argument it can be easily verified that \eqref{Sobolev_bound} yields for intervals $Q=[a,a+\varepsilon]$ of length $\varepsilon>0$ the uniform bound
\begin{equation}\label{Sobolev_bound2}
\|f\|_{L^2(Q)}\leq c_{\alpha}\varepsilon^{\alpha}|f|_{H^{\alpha}(Q)},
\end{equation}
for all $f\in H^{\alpha}(Q)$ vanishing at some $x_0\in Q$. Note that $H^{\alpha}(Q)$ is defined in an analogous way as $H^{\alpha}(\Omega)$. Now with $f_n:=((\mu-\Pi_n\mu)_j\circ F^{-1}_j)_{1\leq j\leq d}$ and $F'_{\min}:=\min_{1\leq j\leq d}\min_{t\in[0,1]}F'_j(t)>0$ the approximation error of $\mu$ satisfies
\begin{equation}\label{MuL2bound}
\|\mu-\Pi_n\mu\|^2_{L^2(\Omega)}\leq\frac{1}{(F'_{\min})^2}\sum^{d}_{j=1}\sum^{n_j}_{i=1}\|(f_n)_j\|^2_{L^2(\I_{n_j,i})}.
\end{equation}
Note that $F^{-1}_j\in C^1_{N'}$ and $(F^{-1}_j)'\in C^{\gamma}_{N''}$ with $N'=(F'_{\min})^{-1}$ and $N''=N(F'_{\min})^{-3}$, which implies that $|(f_n)_j|^2_{H^{\alpha}(I_{n_j,i})}\leq L'_{ij}$ with
\[L'_{ij}=
\begin{cases}\|F'_j\|^2_{\infty}(N')^{2\alpha+1}|(\mu)_j|^2_{H^{\alpha}(I_{n_j,i})},&\alpha\in(1/2,1),\\
N'|(\mu)_j|^2_{H^{\alpha}(I_{n_j,i})},&\alpha=1,\\
2(N'')^2|(\mu)_j|^2_{H^{\alpha}(I_{n_j,i})}+2N''\|(\mu)_j\|^2_{L^2(I_{n_j,i})}n^{-\kappa}_j/\kappa,&\alpha\in(1,2),
\end{cases}\]
where $\kappa:=2(\alpha-\gamma)\in(0,4)$. In particular, $(f_n|_{I_{n_j},i})_j$ lies in $\in H^{\alpha}(\I_{n_j,i})$ having the root $i/n_j,\ 1\leq i\leq n_j,\ 1\leq j\leq d$. Thus by \eqref{Sobolev_bound2}, \eqref{MuL2bound} and the explicit bounds $L'_{ij}$ it follows that
\[\sup_{\mu\in B_{\alpha,M}}\|\mu-\Pi_n\mu\|^2_{L^2(\Omega)}\leq\sup_{\mu\in B_{\alpha,M}}c^2_{\alpha,F}n^{-2\alpha}_{\min}\|\mu\|^2_{H^{\alpha}(\Omega)}\leq c^2_{\alpha,F}L^2n^{-2\alpha}_{\min},\]
where $c_{\alpha,F}$ depends on $\alpha$ and $F$ only. The statement for $\|k-\Pi_n k\|^2_{L^2(\Omega^2)}$ follows analogously, where a similar bound as $L'_{ij}$ for the case $\alpha\in(1,2)$ is used.
\end{proof}

\section{Proofs of parametric results}
\subsection{Proofs for asymptotic information and LAN}
\begin{proof}[\textbf{Proof of Theorem~\ref{ThmRateFisher}}]
Set $\mathfrak D_{np}:=\diag\{\diag\{\mathfrak D^{ij}_{np}\}_{1\leq j\leq d}\}_{1\leq i\leq d}$ with
\begin{equation}\label{Diag}
\mathfrak D^{ij}_{np}:=(s_i+\eta^2/(\lambda_pn))^{-1}(s_j+\eta^2/(\lambda_pn))^{-1},\quad 1\leq i,j\leq d.
\end{equation}
Then $\IIn(\Sigma)=\frac{1}{4}Q^{\otimes2}(\sum^{\infty}_{p=1}\mathfrak D_{np})(Q^{\top})^{\otimes 2}$ and, by $\Sigma\geq S^{-1}I_d$, one has
\begin{equation}\label{IntUniform}
\sum_{p\geq1}\mathfrak D^{ij}_{np}=p_n\int^{\infty}_0 (s_i+\eta^2 x^{\delta})^{-1}(s_j+\eta^2 x^{\delta})^{-1}dx+o(p_n),\quad 1\leq i,j\leq d.
\end{equation}
The equality in \eqref{IntUniform} follows from $\lambda(p_n)=n^{-1}$ along with dominated convergence over sets $(0,y]$ and $[y,\infty)$ under usage of Theorem~\ref{ThmRegVarUni} (2) and (3) applied to $1/\lambda$ and $\lambda$, respectively. In particular, we obtain
\begin{equation}\label{FisherConvergence}
\lim_{n\to\infty}r^2_n\sum^{\infty}_{p=1}\mathfrak D^{ij}_{np}=\zeta\int^{\infty}_0\Big(s_i+\eta^2x^{\delta}\Big)^{-1}\Big(s_j+\eta^2x^{\delta}\Big)^{-1}dx,
\end{equation}
where $\zeta=\lim_{n\to\infty}r^2_np_n$. It is clear that the same limit is already attained for index sets $\pi_n$ as described in the theorem.
\end{proof}

\begin{remark}\label{RmkFisherExp}
For $\delta>1$ and $b\in\N$ the substitution $z=(1+x^{\delta})^{-1}$ gives
\begin{equation}\label{integral}
\int^{\infty}_0\left(1+x^{\delta}\right)^{-b}dx=\frac{1}{\delta}\text{B}\Big(b-\frac{1}{\delta},\frac{1}{\delta}\Big)=\frac{\pi\prod^{b-1}_{j=1}\Big(b-j-\frac{1}{\delta}\Big)}{\delta(b-1)!\sin(\pi/\delta)}.
\end{equation}
where $\prod_{j\in\emptyset}\Big(b-j-\frac{1}{\delta}\Big)=1$ and $\text{B}$ denotes the Beta function. Explicit expressions of \eqref{IntUniform} now follow with \eqref{integral} and observing that for $s_i\neq s_j$
\begin{align*}
&\int^{\infty}_{0}\Big(s_i+\eta^2x^{\delta}\Big)^{-1}\Big(s_j+\eta^2x^{\delta}\Big)^{-1}dx\\
=&\frac{1}{s_i(s_j-s_i)}\int^{\infty}_0(1+x^{\delta}\eta^2/s_i)^{-1}dx-\frac{1}{s_j(s_j-s_i)}\int^{\infty}_0(1+x^{\delta}\eta^2/s_j)^{-1}dx.
\end{align*}
\end{remark}

\begin{proof}[\textbf{Proof of Proposition~\ref{LAN_gen}}]
With $H_n:=r_nH$ it is easy to see that
\begin{align}\label{loglikep1}
\log\frac{dP^n_{\Sigma+H_n}}{dP^n_{\Sigma}}=\sum^{\infty}_{p=1}\Big(&-\frac{1}{2}\log|I_d+\lambda_pC^{-1}_pH_n|\\
&-\frac{1}{2}Y^{\top}_p\Big((C_p+H_n\lambda_p)^{-1}-C^{-1}_p\Big)Y_p\Big).\label{loglikep2}
\end{align}
In the following let $n$ be large enough in the sense that $\lambda_pC_p^{-1}\leq\Sigma^{-1}$ implies
\begin{equation}\label{trace_bound}
|\tr(\lambda_pC^{-1}_pH_n)|\leq\|\Sigma^{-1}\|\|H_n\|<1/2,\ p\geq1.
\end{equation}
A Mercator series expansion applied to the determinant in \eqref{loglikep1} yields
\[-\frac{1}{2}\log(|I_d+\lambda_pC^{-1}_pH_n|)=-\frac{1}{2}\tr(\lambda_pC^{-1}_pH_n)+\frac{1}{4}\tr((\lambda_pC^{-1}_pH_n)^2)-\frac{1}{2}R^{(1)}_{np},\]
where $R^{(1)}_{np}=\frac{1}{2}\sum^{\infty}_{k=3}(-1)^{k+1}\tr((\lambda_pC^{-1}_pH_n)^k)/k$. The term $-\frac{1}{2}\tr(\lambda_pC^{-1}_pH_n)$ is the deterministic part of $\vc(H_n)^{\top}\nabla\ell_p(\Sigma)$ and it holds that
\begin{equation}\label{Fishersign}
\frac{1}{4}\tr((\lambda_pC^{-1}_pH_n)^2)=\frac{1}{2}\vc(H_n)^{\top}\III_{np}(\Sigma)\ZZZ\vc(H_n).
\end{equation}
For $R^{(1)}_n:=\sum^{\infty}_{p=1}R^{(1)}_{np}$ the bound $|\tr(A^{2+k})|^2\leq\tr(A^4)\tr(A^{2k})\leq\tr(A^2)^2\|A\|^{2k},\ A\in\R^{d\times d}_{\text{sym}}$, as well as $\lambda_pC^{-1}_p\leq\Sigma^{-1}$ and \eqref{Fishersign} give
\begin{align*}
|R^{(1)}_n|\leq&\frac{1}{2}\sum^{\infty}_{p=1}\sum^{\infty}_{k=1}|\tr((\lambda_pC^{-1/2}_pH_nC^{-1/2}_p)^{k+2})|\\
\leq&\frac{1}{2}\sum^{\infty}_{p=1}\tr((\lambda_pC^{-1}_pH_n)^2)\sum^{\infty}_{k=1}\|\Sigma^{-1}\|^k\|H_n\|^k\\
\leq&2\vc(H)^{\top}r^2_n\IIn(\Sigma)\ZZZ\vc(H)\|\Sigma^{-1}\|\|H_n\|=\OOO(r_n),
\end{align*}
where Theorem~\ref{ThmRateFisher} was used. Denote the approximation error between \eqref{loglikep2} and the stochastic part of $\vc(H_n)^{\top}\nabla\ell_p(\Sigma)$ by $Y^{\top}_pA_{np}Y_p$, where
\begin{align*}
A_{np}:=
&\frac{1}{2}\Big(C^{-1}_p-(C_p+H_n\lambda_p)^{-1}-\lambda_pC^{-1}_pH_nC^{-1}_p\Big)\\
=&-\frac{\lambda^2_p}{2}(C_p+H_n\lambda_p)^{-1}H_nC^{-1}_pH_nC^{-1}_p.
\end{align*}
Set $\Delta_{n,H}:=\vc(H)^{\top}\ell_n(\Sigma)$, $\rho^{(1)}_n:=\sum^{\infty}_{p=1}(Y^{\top}_pA_{np}Y_p-\E[Y^{\top}_pA_{np}Y_p])$ and $\rho^{(2)}_n:=R^{(2)}_n-\frac{1}{2}R^{(1)}_n$ to obtain
\[\log\frac{dP^n_{\Sigma+H_n}}{dP^n_{\Sigma}}=\Delta_{n,H}-\frac{1}{2}\vc(H_n)^{\top}\IIn(\Sigma)\ZZZ\vc(H_n)+R^{(2)}_n-\frac{1}{2}R^{(1)}_n+\rho^{(2)}_n,\]
where $R^{(2)}_n:=\vc(H_n)^{\top}\IIn(\Sigma)\ZZZ\vc(H_n)+\sum^{\infty}_{p=1}\E[Y^{\top}_pA_{np}Y_p]$ and
\begin{align*}
|R^{(2)}_n|\leq&\sum^{\infty}_{p=1}\frac{\lambda^2_p}{2}|\tr((C^{-1}_p-(C_p+\lambda_pH_n)^{-1})H_nC^{-1}_pH_n)|\\
=&\sum^{\infty}_{p=1}\frac{\lambda^3_p}{2}|\tr((C_p+\lambda_pH_n)^{-1}H_n(C^{-1}_pH_n)^2|\\
\leq&r_n\|\Sigma^{-1}\|\|H\|\vc(H)^{\top}r^2_n\IIn(\Sigma)\ZZZ\vc(H)=\OOO(r_n).
\end{align*}
Using independence of $(Y^{\top}_pA_{np}Y_p-\E[Y^{\top}_pA_{np}Y_p])$, $p\geq1$, it holds that
\begin{align*}
\var(\rho^{(1)}_n)=&\sum^{\infty}_{p=1}\vc(A_{np})^{\top}(C_p\otimes C_p)\ZZZ\vc(A_{np})\\
=&\sum^{\infty}_{p=1}\frac{\lambda^4_p}{2}\tr(H_n(C_p+\lambda_pH_n)^{-1}H_nC^{-1}_pH_n(C_p+\lambda_pH_n)^{-1}H_nC^{-1}_p)\\
\leq&r^2_n\|\Sigma^{-1}\|^2\|H\|^2r^2_n\vc(H)\IIn(\Sigma)\ZZZ\vc(H)=\OOO(r^2_n).
\end{align*}

Let $\pi_n\subseteq\N$ be as in Theorem~\ref{ThmRateFisher}. Then $\nabla\ell_{\pi_n}(\Sigma):=\sum_{p\in\pi_n}\nabla\ell_{np}(\Sigma)$ satisfies $\cov(r_n\ell_{\pi_n}(\Sigma))\to\III(\Sigma)\ZZZ$. Denote by $\III_{\pi_n}(\Sigma)$ the invertible matrix such that $\III_{\pi_n}(\Sigma)^{-1/2}\cov(r_n\ell_{\pi_n}(\Sigma))=\ZZZ$. Then Lyapunav's condition can be verified by bounding 4th moments of Gaussians, which implies Lindeberg's condition. Thus Theorem 5.12 from \citet{Kallenberg[2002]} is applicable and gives
\[r_n\III_{\pi_n}(\Sigma)^{-1/2}\nabla\ell_{\pi_n}(\Sigma)\overset{d}{\to}\NNN(0,\ZZZ).\]
By $\cov(r_n\sum_{p\in\pi^c_n}\nabla\ell_{np}(\Sigma))\to0$ and Slutsky's Lemma the claim follows.
\end{proof}

\subsection{Proof of estimation results}
For $\Sigma\in\Spar$ set $C_p(\Sigma):=\Sigma\lambda_p+\eta^2/nI_d,\ p\geq1$. Then $\lambda_pC_p(\Sigma)^{-1}\leq \Sigma^{-1}$, $\Sigma_p^{-1}\leq \frac{n}{\eta^2}I_d$ and $\|\Sigma\|\leq S\sqrt{d}$ imply
\begin{equation}\label{MVT_summand}
\|\IInp(A)-\IInp(B)\|\leq c_B\|A-B\|\|\IInp(A)\|,\quad p\geq1,
\end{equation}
where $c_B:=2S\sqrt{d}\|B^{-1}\|^2+\sqrt{d}\|B^{-1}\|$, $A,B\in\Spar$. Therefore for any $J\subseteq\N$
\begin{equation}\label{MVT_Fisher}
\|\III_J(A)-\III_J(B)\|\leq c_B\|A-B\|\sum_{p\in J}\|\IInp(A)\|,
\end{equation}
for $A,B\in\Spar$, hence $\|\III_J(\Sigma)^{-1}\|\leq\|\III_J(SI_d)^{-1}\|$, $\Sigma\in\Spar$, yields
\begin{equation}\label{MVT_FisherInverse}
\|\III_J(A)^{-1}-\III_J(B)^{-1}\|\leq c_B\|A-B\|\|\III_J(SI_d)^{-1}\|^2\sum_{p\in J}\|\IInp(A)\|.
\end{equation}
\begin{proof}[\textbf{Proof of Theorem~\ref{ThmOrAd}}]
In the following we always consider the measure $P^n_{\Sigma+r_nH}$. By construction, $\hat\vartheta^{\text{or}}_n$ is unbiased. It can be easily seen that
\[\cov(\hat\vartheta^{\text{or}}_n)=\frac{1}{4}\IItr(\Sigma+r_nH)^{-1}\ZZZ.\]
In analogy to \eqref{IntUniform} and \eqref{FisherConvergence} note that also
\begin{equation}\label{Fishernorm_convergence}
r^2_n\sum_{p\in\pi_n}\|\IInp(\Sigma)\|=\OOO(1).
\end{equation}
Thus \eqref{MVT_FisherInverse} and $r^2_n\IItr(SI_d)\to\III(SI_d)$ (cf. Theorem~\ref{ThmRateFisher}) imply
\[r^{-2}_n\|\cov(\hat\vartheta^{\text{or}}_n)-\tfrac{1}{4}\IItr(\Sigma)^{-1}\ZZZ\|=\OOO(r_n)\]
and the central limit theorem can be deduced as in Proposition~\ref{LAN_gen}.

With the bound $\vc(\hat\Sigma^{\text{pre}}_n)$ and $\IInp(\Sigma+r_nH)^{-1}\leq\IInp(SI_d)^{-1}$ we obtain
\[\|\cov(\hat\vartheta^{\text{pre}}_n-\psi(\Sigma+r_nH))\|\leq\frac{1}{4}\|\III_{\pi'_n}(SI_d)^{-1}\|=\OOO(r^2_n).\]
Since $\vc(\hat\Sigma^{\text{pre}})$ is unbiased, there is some $\varepsilon_n\to0$ with $r_n=o(\varepsilon_n)$, as $n\to\infty$, and such that the events $E_n:=\{\|\hat\Sigma^{\text{pre}}_n-(\Sigma+r_nH)\|\leq\varepsilon_n\}$, $n\geq1$, satisfy
\begin{equation}\label{preestasymp}
P^n_{\Sigma+r_nH}(E^c_n)=o(1).
\end{equation}
By Slutsky's Lemma, the claim for $\hat\vartheta^{\text{ad}}_n$ follows if $\|\hat\vartheta^{\text{or}}_n-\hat\vartheta^{\text{ad}}_n\|=o_{P^n_{\Sigma+r_nH}}(r_n)$. First note that there is some $S_0>0$ such that on $E_n$ it holds that $\hat\Sigma^{\text{pre}}_n\in\Spar$ and $\hat\Sigma^{\text{pre}}_n>S_0I_d$, for all $n\geq n^*$, with $n^*$ sufficiently large. Thus $\|\IItr(A)^{-1}\|=\OOO(r^2_n)$, $A\in\Spar$, \eqref{MVT_summand}, \eqref{MVT_Fisher}, \eqref{MVT_FisherInverse} and \eqref{Fishernorm_convergence} yield
\begin{equation}\label{weightbound}
\|W_p-\hat W_p\|\1_{E_n}\leq c'\varepsilon_n\|\IItr(SI_d)^{-1}\|\|\IInp(\Sigma+r_nH)\|,\quad p\in\pi_n,
\end{equation}
where $W_p=W_p(\Sigma+r_nH)$, $\hat W_p:=W_p(\hat\Sigma^{\text{pre}}_n)$ and $c'$ depends only on $S$ and $S_0$. Independence by $\pi_n\cap\pi'_n=\emptyset$, $\|\IItr(SI_d)^{-1}\|=\OOO(r^2_n)$ and \eqref{Fishernorm_convergence} imply
\begin{align*}
\E[\tr((\hat\vartheta^{\text{or}}_n-\hat\vartheta^{\text{ad}}_n)(\hat\vartheta^{\text{or}}_n-\hat\vartheta^{\text{ad}}_n)^{\top}\1(E_n))]=\frac{1}{4}&\sum_{p\in\pi_n}\tr\Big(\IInp(\Sigma+r_nH)^{-1}\ZZZ\\
&\cdot\E[(W_p-\hat W_p)^2\1(E_n)]\Big)=\OOO(r^2_n\varepsilon^2_n).
\end{align*}
This along with Markov's inequality now gives
\[P^n_{\Sigma+r_nH}(r^{-1}_n\|\hat\vartheta^{\text{or}}_n-\hat\vartheta^{\text{ad}}_n\|\geq\varepsilon)=P^n_{\Sigma+r_nH}(E^c_n)+\OOO(\varepsilon^2_n\varepsilon^{-2}),\]
which implies $\|\hat\vartheta^{\text{or}}_n-\hat\vartheta^{\text{ad}}_n\|=o_{P^n_{\Sigma+r_nH}}(r_n)$, by \eqref{preestasymp} and $\varepsilon_n\to0$.
\end{proof}

\subsection{Proofs of further asymptotic equivalences}
\begin{proof}[\textbf{Proofs of Proposition~\ref{PropLeCamProj}}]
Let $\hat\Sigma_n=\text{mat}(\hat\vartheta^{\text{ad}}_n)$ be the adaptive estimator induced by \eqref{adaptest}. Note that $\pi_n$ can be split up into two disjoint sets $\pi'_n$ and $\pi''_n$ such that the underlying pre-estimator $\hat\Sigma^{\text{pre}}_n$ is build on $(Y_p)_{p\in\pi'_n}$ and $\hat\vartheta^{\text{ad}}_n=\sum_{p\in\pi''_n}W_p(\hat\Sigma^{\text{pre}}_n)\hat\vartheta_p$. For i.i.d. $Z_p\sim\NNN(0,I_d)$ set $\hat Y_p:=(\hat\Sigma_n\lambda_p+\frac{\eta^2}{n}I_d)^{1/2}Z_p$, $p\in\pi^c_n:=\N\backslash\pi_n$. Then, given $(Y_p)_{p\in\pi_n}$, $\hat Y_p$ is centred Gaussian with covariance $\hat\Sigma_n\lambda_p+\frac{\eta^2}{n}I_d$. Let $Y:=(Y_p)_{p\geq1}$ and $\hat Y:=(Y_p)_{p\in\pi_n}\cup(\hat Y_p)_{p\in\pi^c_n}$. Now observe that (with $P^{\pi_n}_{\Sigma}:=\LLL((Y_p)_{p\in\pi_n}$)
\begin{align*}
&H^2(\LLL(Y),\LLL(\hat Y))\\
\leq&\E^{\pi_n}_{\Sigma}[H^2(\LLL(Y|(Y_p)_{p\in\pi_n}),\LLL(\hat Y|(Y_p)_{p\in\pi_n}))\1_{A_n}]+2P^{\pi_n}_{\Sigma}(A^c_n),
\end{align*}
where $A_n:=\{\|\hat\Sigma'_n-\Sigma\|\leq v_n\}$, $v_n:=Cr_n\log(n)$ and $C>0$ specified below. With \eqref{hellingergauss} and by regular variation of $\lambda$ (cf. proof of Theorem~\ref{ThmRateFisher}) deduce
\begin{align*}
&\E^{\pi_n}_{\Sigma}[H^2(\LLL(Y|(Y_p)_{p\in\pi_n}),\LLL(\hat Y|(Y_p)_{p\in\pi_n}))\1(A_n)]\\
\lesssim&v^2_n\sum_{p\in\pi^c_n}(S^{-1}+\frac{\eta^2}{\lambda_pn})^{-2}\lesssim C^2S^{-1\delta}\eta^{-2/\delta}\log^2(n)\int_{(a_n,b_n)^c}(1+x^{\delta})^{-2}dx\\
\lesssim&C^2S^{-1\delta}\eta^{-2/\delta}\log^2(n)(a^{2\delta-1}_n\vee b^{1-2\delta}_n),
\end{align*}
uniformly in $\Sigma\in\Spar'$. Moreover, $\sup_{\Sigma\in\Spar'}P^{\pi_n}_{\Sigma}(A^c_n)=\OOO(n^{-1})$ can be easily shown by a Fuk-Nagaev type inequality such as Theorem~3.1 from \cite{Einmahl[2008]}, where we choose $\xi_p:=r^{-1}_nW_p(SI_d)(\hat\vartheta_p-\vartheta)$ and $C$ such that $(C-1)^2\geq3\sum_{p\in\pi_n}\E[\|\xi_p\|^2]$. With \eqref{lecamhell} the claim follows.
\end{proof}

\begin{proof}[\textbf{Proof of Proposition~\ref{PropEigLeCamEqui}}]
(ii) implies (i): By assumption, $\III(\Sigma)=\III'(\Sigma)$ and $r_n/r_n'\to1$, which implies $p_n/p'_n\to1$. Due to uniform convergence of $\lambda'(\cdot\ p)/\lambda'(p)$ on $[\inf_np_n/p'_n,\infty)\subsetneq(0,\infty)$ (cf. Theorem \ref{ThmRegVarUni} (2)) it follows that $\lambda(p_n)/\lambda'(p_n)=\lambda'(p_n')/\lambda'(p_n)\to1$. Therefore, for any $p>0$
\begin{equation}\label{ConvSubseq}
\lim_{n\to\infty}\frac{\lambda(p_n\cdot p)}{\lambda'(p_n\cdot p)}=\lim_{n\to\infty}\frac{\lambda(p_n\cdot p)}{\lambda'(p_n\cdot p)}\frac{\lambda'(p_n)}{\lambda(p_n)}=1.
\end{equation}
By Proposition \ref{PropRegVar} (3) $g:=\lambda/\lambda'$ is slowly varying, hence the convergences in \eqref{ConvSubseq} hold uniformly over any $[a,b]\subset\R_+$, cf. Theorem \ref{ThmRegVarUni} (1). Thus for any $\varepsilon>0$ and reals $0<x<y<\infty$ there is some $n^*\in\N$ with $|g(p)-1|\leq\varepsilon$, for all $p\in[p_nx,p_ny],\ n\geq n^*$. Since there is some $N\geq n^*$ such that
\[[p_nx,\infty)=\bigcup_{n\geq N}[p_nx,p_ny],\]
we have for any $p\geq p_nx$ that $|g(p)-1|\leq\varepsilon$, i.e. $\lambda/\lambda'\to1$.

(i) implies (iii): Assume $\lambda_p/\lambda'_p\to1$ and denote by $\FFF^{s,\text{even}}_n$ and $\FFF^{s,\text{odd}}_n$ the experiments that are generated by $(Y_p)_{p\in2\N}$ and $(Y_p)_{p\in2\N+1}$, respectively, such that $\ExpFS=\FFF^{s,\text{even}}_n\otimes\FFF^{s,\text{odd}}_n$. Similarly, obtain the decomposition $\FFF^{s'}_n=\FFF^{s',\text{even}}_n\otimes\FFF^{s',\text{odd}}_n$. Since $\Delta$ satisfies the triangle inequality it suffices to show that both, $\ExpFS$ and $\FFF^{s'}_n$ are asymptotically equivalent to the experiment
\[\EEE_n:=\FFF^{s',\text{even}}_n\otimes\FFF^{s,\text{odd}}_n,\]
which will be shown under the localisation approach of \citet{Grama[2002]}. Let $v_n:=Cr_n\log(n)$ with $C$ specified below and denote for fixed $\Sigma\in\Spar'$ by $\FFF^s_{n,\loc}$ and $\FFF^{s'}_{n,\loc}$ the localisations that are generated by
\begin{align*}
Y_p&\sim\NNN(0,(\Sigma+v_nH)\lambda_p+\tfrac{\eta^2}{n}),\quad p\geq1,\\
Y'_p&\sim\NNN(0,(\Sigma+v_nH)\lambda'_p+\tfrac{\eta^2}{n}),\quad p\geq1,
\end{align*}
respectively, where $H\in B_r(\Sigma):=\{A\in\Rddsym:\|A-\Sigma\|\leq r\}$ is unknown, $r>0$. Similarly introduce the local experiments $\FFF^{s,\text{even}}_{n,\loc},\FFF^{s',\text{even}}_{n,\loc}$ and $\FFF^{s,\text{odd}}_{n,\loc},\FFF^{s',\text{odd}}_{n,\loc}$ that are generated by the even and odd indices, respectively.

Comparing $\FFF^{s,\text{even}}_{n,\loc}$ with $\FFF^{s',\text{even}}_{n,\loc}$ it is evident that whenever $\lambda_p\geq\lambda'_p$ then $Y_p$ is at least as informative as $Y'_p$. To see this consider the equivalent normalisation $\lambda^{-1/2}_pY_p$ and add uninformative noise to match $(\lambda'_p)^{-1/2}Y'_p$ in law. Therefore, without loss of generality we assume that $\lambda_p<\lambda'_p,\ \forall p\in2\N$. By adding uninformative and independent $\NNN(0,\Sigma(\lambda'_p-\lambda_p))$-noise to $Y_p$ we obtain the independent sequence
\[Y''_p\sim\NNN(0,\Sigma\lambda'_p+v_nH\lambda_p+\tfrac{\eta^2}{n}I_d),\ p\in2\N.\]
For $a_n\to0$ such that $a_np_n\to\infty$ let $\pi_n:=[a_np_n,\infty)\cap2\N$. Then Proposition~\ref{PropLeCamProj}, the Hellinger bound \eqref{hellingergauss} and $S^{-1}I_d<\Sigma$ give
\begin{equation}\label{hellbound2}
H^2(\LLL((Y''_p)_{p\in2\N}),\LLL((Y'_p)_{p\in2\N}))\lesssim(\log^2(n)a^{2\delta-1}_n)\vee\Big(\sum_{p\in\pi_n}\frac{v^2_n(\lambda_p-\lambda'_p)^2}{(S^{-1}\lambda_p+\tfrac{\eta^2}{n})^2}\Big).
\end{equation}
By an integral approximation (cf. the proof of Theorem~\ref{ThmRateFisher}) it follows with $\tau(p):=|\lambda(p)-\lambda'(p)|/\lambda(p)=o(1)$ (for $p\to\infty$) that
\begin{align*}
v^2_n\sum_{p\in\pi_n}\frac{(\lambda_p-\lambda'_p)^2}{(S^{-1}\lambda_p+\tfrac{\eta^2}{n})^2}\lesssim&\tau(a_np_n)^2v^2_n\sum_{p\in\pi_n}(S^{-1}+\eta^2/(\lambda_pn))^{-2}\\
=&\OOO(C^2S^{2-1/\delta}\eta^{-2/\delta}r^2\tau(a_np_n)^2\log(n)^2),
\end{align*}
which along with \eqref{lecamhell} and \eqref{hellbound2} implies
\begin{equation}\label{localbound}
\sup_{\Sigma\in\Spar'}\Delta(\FFF^{s,\text{even}}_{n,\loc},\FFF^{s',\text{even}}_{n,\loc})=o(r\tau(a_np_n)\log(n)C^2S^{1-1/2\delta}\eta^{-1/\delta}).
\end{equation}
Analogously to the proof of Proposition~\ref{PropLeCamProj} it is possible to construct a consistent estimator $\hat\Sigma'_n$ of $\Sigma$ in $\FFF^{s,\text{odd}}_n$ to obtain for $A_n=\{\|\hat\Sigma'_n-\Sigma\|\leq v_n\}$
\[\sup_{\Sigma\in\Spar'}P^n_{\Sigma}(A^c_n)=o(1).\]
Since $P^n_{\Sigma}(A^c_n)$ can be controlled uniformly in $\Sigma\in\Spar'$, and since under the event $A_n$ the bound \eqref{localbound} applies, one has $\Delta(\FFF^s_n,\EEE_n)=o(1)$. In the same way $\Delta(\FFF^{s'}_n,\EEE_n)=o(1)$ can be obtained and (i) follows.

(iii) implies (ii): Assume $\Delta(\FFF^s_n,\FFF^{s'}_n)\to0$. Then for $\Sigma\in\Spar'$ fixed also any pair of local sub experiments satisfies $\Delta(\FFF^{s,\text{even}}_{n,\loc}\otimes\FFF^{s,\text{even}}_{n,\loc},\FFF^{s',\text{even}}_{n,\loc}\otimes\FFF^{s',\text{even}}_{n,\loc})\to0$ and therefore, by Proposition~\ref{LAN_gen}, both experiments satisfy the same LAN-expansion with $r_n/r'_n\to1$ and $\III(\Sigma)=\III'(\Sigma)$.
\end{proof}

\section{Proofs of semi-parametric results}
\subsection{Piecewise constant approximation}
\begin{proof}[\textbf{Proof of Proposition~\ref{PropPwc}}]
Let $\ExpMCTwo$ be the statistical experiment that is generated by observing
\[dY_t=X_tdt+\Psi_n(t)dW_t,\quad t\in[0,1],\]
where $\Psi^2_n:=\diag(\eta^2_j/(n_jF'_j))_{1\leq j\leq d}$. Note that the covariance function of $X$ lies in $H^{3/2-\kappa}$, for any $\kappa\in(0,3/2)$ as it has weak derivatives that are continuous on $[0,1]$ except for a single jump. Thus, given that Assumption~\ref{AssRegF}-$F(\gamma,N,\beta)$ is met for $\gamma>\beta$, Theorem~\ref{ThmLeCamDiscCont} gives (with $\beta'=1+\beta$)
\[\Delta(\ExpM,\ExpMCTwo)=\OOO(SMn_{\max}n^{-3/2+\kappa}_{\min}).\]

Next $\Delta(\ExpMC,\ExpMCTwo)=o(1)$ is shown. Introduce the $L^2([0,1],\R^d)$-operators $T_{\Sigma}:f\mapsto \Sigma f,\ R:f\mapsto-\int^1_{\cdot}f(s)ds$ and $R^*:f\mapsto-\int^{\cdot}_0f(s)ds$. Then
\[K'_{\Sigma,n}:=K_{\Sigma}+T_{\Psi^2_n}=R^*T_{\Sigma}R+T_{\Psi^2_n}\]
is the covariance operator in $\ExpMCTwo$. Note that $K^d_{\text{BM}}:=R^*R$ is just the covariance operator of the $d$-dimensional standard Brownian motion. Let further
\[K_{\Sigma,n}:=K_{\Sigma_m}+T_{\Xi^2_{m}}\]
be the covariance operator in $\ExpMC$. Since $T_{\Psi^2_n}-T_{\Xi^2_{m}}$ is not Hilbert-Schmidt in most cases, it is meaningful to consider the one-to-one transformations
\[\tilde K'_{\Sigma,n}:=T_{\Psi^{-1}_n}K'_{\Sigma,n}T_{\Psi^{-1}_n}\quad\text{and}\quad\tilde K_{\Sigma,n}:=T_{\Xi^{-1}_{m}}K_{\Sigma,n}T_{\Xi^{-1}_{m}}.\]
Then \eqref{lecamhilbert} yields the following bound for the Le Cam distance
\begin{equation}\label{LeCamBoundMCMCm}
\Delta(\ExpMC,\ExpMCTwo)\lesssim\sup_{\Sigma\in\Snon}\|(\tilde K'_{\Sigma,n})^{-1/2}(\tilde K'_{\Sigma,n}-\tilde K_{\Sigma,n})(\tilde K'_{\Sigma,n})^{-1/2}\|_{\text{HS}}.
\end{equation}
Let $\tau_n:=n_{\min}\min_{1\leq j\leq d}\min_{t\in[0,1]}F'_j(t)/\eta^2_j$. With $S^{-1}I_d<\Sigma$ and $\tau^{1/2}_nI_d\leq T_{\Psi^{-1}_n}$ we obtain
\begin{equation}\label{covbound}
\tau_n S^{-1}K^d_{\text{BM}}+\Id<\tilde K'_{\Sigma,n}.
\end{equation}
Let further $\varphi_p(t):=\sqrt{2}\sin((p-1/2)\pi t),\ p\geq 1$, be the eigenbasis of $K^1_{\text{BM}}$ and $e_{ip}:=(\1_{\{i=j\}}\varphi_i)_{1\leq j\leq d},\ i=1,\ldots,d,p\geq1,$ be a basis of $L^2([0,1],\R^d)$. For $A=\diag(a_1,\ldots,a_d):[0,1]\to\R^{d\times d}_+$ integration by parts yields
\begin{align*}
|\langle T_AK_{\Sigma}T_Ae_{ip},e_{ip}\rangle|=&\Big|\sum^d_{l=1}\int_{[0,1]^2}\Big(\int^{s\wedge t}_0(\Sigma(u))_{li}du\Big)a_i(s)\varphi_p(s)a_i(t)\varphi_p(t)dsdt\Big|\\
=&\Big|\sum^d_{l=1}\int_{[0,1]}(\Sigma(u))_{li}\Big(-\int^1_ua_i(s)\varphi_p(s)\Big)^2du\Big|\\
=&\langle \Sigma E_{A,ip},E_{A,ip}\rangle.
\end{align*}
Here $E_{A,ip}(t):=-\int^1_tA(s)e_{ip}(s)ds$ satisfies with $\lambda_p:=((p-1/2)\pi)^{-2}$
\begin{equation}\label{supnormbound}
\|E_{A,ip}\|_{\infty}\leq\sqrt{2\lambda_p}\|A\|_{\infty},
\end{equation}
because $s\mapsto-\sqrt{2}\cos(\lambda^{-1/2}_ps)\text{sgn}(\varphi_p(s))\lambda^{1/2}_p$ is the anti-derivative of $|\varphi_p(s)|$.
For $B:=\Psi^{-1}_n$ and $C:=\Xi^{-1}_{m}$ the bound $\|F'_j-F'_{j,m}\|_{\infty}\leq Nm^{-\gamma}$ implies
\begin{equation}\label{noiselevelgap}
\|(B-C)_{jj}\|_{\infty}=n^{1/2}_j\eta^{-1}_j\Big\|\sqrt{F'_j}-\sqrt{F'_{j,m}}\Big\|_{\infty}\leq n^{1/2}_{\max}m^{-\gamma}N\nu^{-1}_j\|(F'_j)^{-1/2}\|_{\infty},
\end{equation}
for any $j=1,\ldots,d$. Moreover
\begin{align}
\nonumber&|\langle (T_BK_{\Sigma}T_B-T_CK_{\Sigma_m}T_C)e_{ip},e_{ip}\rangle|\\
\leq&|\langle \Sigma E_{B,ip},E_{B-C,ip}\rangle|+|\langle (\Sigma-\Sigma_m) E_{B,ip},E_{C,ip}\rangle|+|\langle \Sigma_m E_{C,ip},E_{B-C,ip}\rangle|\label{covopbound}
\end{align}
Now by $\|C\|_{\infty}\leq\|B\|_{\infty}\leq n^{1/2}_{\max}\max_{1\leq j\leq d}\|(F'_j)^{1/2}\|_{\infty}/\eta^2_j$, \eqref{supnormbound} and \eqref{noiselevelgap}
\begin{align}
|\langle \Sigma E_{B,ip},E_{B-C,ip}\rangle|+|\langle \Sigma_m E_{C,ip},E_{B-C,ip}\rangle|&\lesssim Sn_{\max}m^{-\gamma}N\lambda_p,\label{partialone}\\
|\langle (\Sigma-\Sigma_m) E_{B,ip},E_{C,ip}\rangle|&\lesssim n_{\max}m^{-\beta}M\lambda_p,\label{partialtwo}
\end{align}
uniformly in $\Sigma\in\Snon$ and $p\geq1$, where the Sobolev-bound $\sup_{\Sigma\in\Snon}\|\Sigma-\Sigma_m\|_{L^2(\R^{d\times d})}=\OOO(Mm^{-\beta})$ has been used, cf. the proof of Theorem~\ref{ThmLeCamDiscCont}. Thus \eqref{covopbound}, \eqref{partialone}, \eqref{partialtwo} and $\lambda_p=\langle K^d_{\text{BM}}e_{ip},e_{ip}\rangle$ imply
\begin{equation}\label{boundabove}
|\langle(\tilde K'_{\Sigma,n}-\tilde K_{\Sigma,n})e_{ip},e_{ip}\rangle|\lesssim S(N\vee M)n_{\max}m^{-(\beta\wedge\gamma)}\langle K^d_{\text{BM}}e_{ip},e_{ip}\rangle
\end{equation}
By \eqref{covbound} and \eqref{boundabove} the right-hand side of \eqref{LeCamBoundMCMCm} can be bounded by
\begin{equation}\label{partialthree}
S(N\vee M)n_{\max}m^{-(\beta\wedge\gamma)}\tau^{-1}_n\|(S^{-1}K^d_{\text{BM}}+\tau^{-1}_n\Id)^{-1}K^d_{\text{BM}}\|_{HS}.
\end{equation}
Applying $(\tau_nS^{-1}K^d_{\text{BM}}+\Id)^{-1}K^d_{\text{BM}}$ to the basis $(e_{ip})_{i=1,\ldots,d;p\geq1}$ yields
\begin{equation}\label{partialfour}
\|(S^{-1}K^d_{\text{BM}}+\tau^{-1}_n\Id)^{-1}K^d_{\text{BM}}\|^2_{\text{HS}}=d\sum^{\infty}_{p=1}(S^{-1}+\lambda^{-1}_p\tau^{-1}_n)^{-2}=\OOO(\tau^{1/2}_n),
\end{equation}
cf. the proof of Theorem~\ref{ThmRateFisher} for $\delta=2$. Then $\Delta(\ExpMC,\ExpMCTwo)=\OOO(Sn^{1/4}_{\max}(Mm^{-\beta}\vee Nm^{-\gamma}))$ follows by \eqref{LeCamBoundMCMCm}, \eqref{partialthree} and \eqref{partialfour}.
\end{proof}

\subsection{Proof of Proposition~\ref{PropLANequi}}
Denote by $\ExpMS=\Big\{Q^n_{\Sigma}:\Sigma\in\Snon\Big\}$, $\ExpMSS:=\Big\{R^n_{\Sigma}:\Sigma\in\Snon\Big\}$ and $\ExpMSSTwo:=\Big\{R^{n'}_{\Sigma}:\Sigma\in\Snon\Big\}$ the statistical experiments that are generated by \eqref{SeqSemi}, $(S_{pk})_{p\geq 0,k}$ and $(S_{pk})_{p\geq 1,k}$, respectively, where $S_{pk}$ is as in \eqref{SeqSemiTwo}. For $\Sigma\in\Snon$ fixed and $r>0$ denote the corresponding localisations by
\[\ExpMSSloc=\Big\{R^n_{\Sigma+n^{-1/4}_{\min}H}:H\in B^{\beta}_{r,n}(\Sigma)\Big\},\quad
\ExpMSSTwoloc=\Big\{R^{n'}_{\Sigma+n^{-1/4}_{\min}H}:H\in B^{\beta}_{r,n}(\Sigma)\Big\},\]
\[\ExpMSloc=\Big\{Q^n_{\Sigma+n^{-1/4}_{\min}H}:H\in B^{\beta}_{r,n}(\Sigma)\Big\},\]
respectively, where $B^{\beta}_{r,n}(f):=\{f+n^{-1/4}_{\min}g\in\HBsym:\|g\|_{\infty}\leq r\}$. Then, by $\Delta(\ExpM,\ExpMC)=o(1)$ and $\Delta(\ExpMC,\ExpMSS)=0$, Proposition~\ref{PropLANequi} follows if
\begin{equation}\label{LeCamSeqs}
\Delta\Big(\ExpMSSloc,\ExpMSloc\Big)=o(1).
\end{equation}
In fact, \eqref{LeCamSeqs} is implied by the following two Lemmas, given $m=o(\sqrt{n_{\min}})$.

\begin{lemma}
For any $r>0$ it holds that
\[\Delta(\ExpMSSloc,\ExpMSSTwoloc)=\OOO\Big(rmn^{-1/2}_{\min}\Big).\]
\end{lemma}
\begin{proof}
The anti-derivatives $\Phi_{p,k}(t)=-\int^1_t\varphi_{p,k}(s)ds$ satisfy
\[\Phi_{0,0}=2\sum^{\infty}_{p=1}\Phi_{p,0},\quad \Phi_{0,k}=\sum^{\infty}_{p=1}(-1)^{p+1}\Phi_{p,k-1}+\Phi_{p,k}\]
and the signals $X_{0,k}$ of $S_{0,k},\ k=0,\ldots,m-1,$ can be represented by
\[X^H_{0,0}:=2\sum^{\infty}_{p=1}X^H_{p,0},\quad X^H_{0,k}:=\sum^{\infty}_{p=1}(-1)^{p+1}X^H_{p,k-1}+X^H_{p,k}\]
with $X^H_{p,k}=(-\int^1_0E^{\top}_{pkj}(t)(\Sigma^H)^{1/2}_m(t)dB_t)_{1\leq j\leq d}$. Note that $\LLL(X^H_{pk}|S_{pk})=\NNN\Big(M^H_{pk}S_{pk},V^H_{pk}\Big)$ with
\[M^H_{pk}:=\Sigma^H_{pk}(\Sigma^H_{pk}+\Xi^2_{m,k})^{-1},\quad V^H_{pk}:=((\Sigma^H_{pk})^{-1}+\Xi^{-2}_{m,k})^{-1},\]
\[\Sigma^H_{pk}:=(\Sigma+n^{-1/4}_{\min}H)(k/m)\lambda_{pm},\]
where $\Xi^2_{m,k}:=\Xi^2_m(k/m)$. For an i.i.d. sequence $Z_{pk}\sim\NNN(0,I_d),\ k=0,\ldots,m-1,p\geq1$, independent of $(S_{pk})_{p\geq1,k}$, construct the signals $X^0_{p,k}:=M^0_{pk}S_{pk}+(V^0_{pk})^{1/2}Z_{pk}$ and create the corresponding $X^0_{0,k}$ by plugging in via
\[X^0_{0,0}:=2\sum^{\infty}_{p=1}X^0_{p,0},\quad X^0_{0,k}:=\sum^{\infty}_{p=1}(-1)^{p+1}X^0_{p,k-1}+X^0_{pk}.\]
With an independent $d$-dimensional Brownian motion $W'$ set
\[S'_{0,k}:=X^0_{0,k}+\Big(\int^1_0\varphi^{\top}_{0,k}(t)\Xi_m(t)dW'_t\Big)_{1\leq j\leq d},\quad k=1,\ldots,m-1.\]
Conditioned on $(S_{pk})_{p\geq1,k}$ the expectation of $S_{0,k}$ is given by
\[m^H_0:=2\sum^{\infty}_{p=1}M^H_{p,0}S_{p,0},\quad m^H_k:=\sum^{\infty}_{p=1}(-1)^{p+1}M^H_{p,k-1}S_{p,k-1}+M^H_{pk}S_{pk},\]
for $k=1,\ldots,m-1$, and the conditional covariance $K^H$ of $(S_{0,k})_{k=0,\ldots,m-1}$ is a $\R^{dm\times dm}$-triangular block matrix with block diagonal
\begin{align*}
K^H_{0,0}&:=2\sum^{\infty}_{p=1}V^H_{p,0}+\Xi^2_{m,0},\\
K^H_{k,k}&:=\sum^{\infty}_{p=1}V^H_{p,k-1}+V^H_{pk}+\frac{1}{2}(\Xi^2_{m,k-1}+\Xi^2_{m,k}),\quad k=1,\ldots,m-1,
\end{align*}
and lower and upper block diagonal
\begin{align*}
K^H_{0,1}&:=2\sum^{\infty}_{p=1}(-1)^{p+1}V^H_{p,0}+\frac{1}{\sqrt{2}}\Xi^2_{m,0},\\
K^H_{k,k+1}&:=\sum^{\infty}_{p=1}(-1)^{p+1}V^H_{pk}-\frac{1}{2}\Xi^2_{m,k},\quad k=1,\ldots,m-2.
\end{align*}
For $S'_{0,k}$ the conditional mean and covariance are given by $m^0_k$ and $K^0$, respectively. Regular variation of $\lambda_{pm}$ as in the proof of Theorem~\ref{ThmRateFisher} yield
\[\sum^{\infty}_{p=1}V_{pk}=\sum^{\infty}_{p=1}((\Sigma^H_{pk})^{-1}+\Xi^{-2}_{m,k})^{-1}\sim\frac{1}{\sqrt{n_{\min}}m}I_d,\quad \sum^{\infty}_{p=1}(-1)^{p+1}V_{pk}\sim I_d,\]
hence $n^{-1/2}_{\min}m^{-1}I_{dm}\lesssim K^0$. Let $S:=(S_{pk})_{k=0,\ldots,m-1,p\geq0}$ and $S':=(S'_{0,k})_{k=0,\ldots,m-1}\cup(S_{pk})_{k=0,\ldots,m-1,p\geq1}$. Then conditioning on $(S_{pk})_{p\geq1,k}$ along with \eqref{hellingergauss} yields
\begin{align}
\nonumber H^2(\LLL(S'),\LLL(S'))&=\E_{p\geq1,k}[H^2(\LLL(S|(S_{pk})_{p\geq1,k}),\LLL(S'|(S_{pk})_{p\geq1,k}))]\\
\label{hellbound3}&\lesssim\sqrt{n_{\min}}m\E_{p\geq1,k}[\|m^H-m^0\|^2]+n_{\min}m^2\|K^H-K^0\|^2_{\text{HS}}
\end{align}
If the bounds $S^{-1}I_d<\Sigma<SI_d$ and the same calculations as for Theorem~\ref{ThmRateFisher} are used it is not hard to see that for any $k=0,\ldots,m-1$
\begin{align*}
\E_{p\geq1,k}[\|m^H_k-m^0_k\|^2]&\lesssim\sum^{\infty}_{p=1}\|(M^H_{p,k}-M^0_{p,k})^{\top}(\Sigma^H_{pk}+\Xi^2_{m,k})^{1/2}\|^2_{\text{HS}}\\
&\lesssim r^2n^{-3/2}_{\min}\sum^{\infty}_{p=1}\|\lambda^{1/2}_{pm}(\Sigma_{p,k}+\Xi^2_{m,k})^{-1/2}\|^2_{\text{HS}}\lesssim \frac{r^2}{n_{\min}m},
\end{align*}
uniformly in $H\in B^{\beta}_{r,n}$, as well as
\begin{align*}
\|K^H-K^0\|^2_{HS}&\lesssim r^2\sum^{m-1}_{k=0}\sum^{\infty}_{p=1}\|V^H_{pk}-V^0_{pk}\|^2_{\text{HS}}\\
&\lesssim n^{-5/2}_{\min}r^2\sum^{m-1}_{k=0}\sum^{\infty}_{p=1}\|\lambda_{pm}(\Sigma_{m,k}\nu_p+\Xi^2_{m,k})\|^2_{\text{HS}}\lesssim r^2n^{-2}_{\min},
\end{align*}
uniformly in $H\in B^{\beta}_{r,n}$. Thus \eqref{lecamhell} and \eqref{hellbound3} yield the claim.
\end{proof}

\begin{lemma}
For any $r>0$ it holds that
\[\Delta(\ExpMSSTwoloc,\ExpMSloc)=\OOO(r\xi_n),\]
where $\xi_n:=\max_{1\leq j,j'\leq d}|n_{\min}/\sqrt{\nu_j\nu_{j'}}-\sqrt{n_jn_{j'}}|=o(1)$.
\end{lemma}
\begin{proof}
Let $J:=\{j:(K_n)_{jj}\geq(K'_n)_{jj}\}$. Then, for $k=1,\ldots,m,\ p\geq1$, the observations in $\ExpMSloc$ and $\ExpMSSTwoloc$ are given by
\begin{align*}
Y_{pk}\sim&\NNN\Big(0,\lambda_{mp}(\Sigma+n^{-1/4}_{\min}H)(k/m)+G(k/m)K^2_n\Big),\\
S_{pk}\sim&\NNN\Big(0,\lambda_{mp}(\Sigma+n^{-1/4}_{\min}H)(k/m)+G(k/m)(K'_n)^2\Big),
\end{align*}
respectively, where $K^2_n:=\diag(\nu_j/n_{\min})_{1\leq j\leq d},\ (K'_n)^2:=\diag(n^{-1}_j)_{1\leq j\leq d}$ and $G:=\diag(\eta^2_j/F'_{j,m})_{1\leq j\leq d}$. As in the proof of Proposition~\ref{PropPwc} consider equivalent one-to-one (covariance) transformations. More precisely, set
\[Y'_{pk}:=K^{-1}_nY_{pk},\quad k=1,\ldots,m,\ p\geq1,\]
which generates an experiment $\ExpMSTwoloc$ equivalent to $\ExpMSloc$. For i.i.d. vectors $Z_{pk}\sim\NNN(0,I_d),\ k=1,\ldots,m,\ p\geq1$, independent of $(S_{pk})_{k=1,\ldots,m,\ p\geq1}$,
set
\[S'_{pk}:=(K''_n)^{-1}(S_{pk}+(G(k/m)W_n)^{1/2}Z_{pk}),\quad k=1,\ldots,m,\ p\geq1\]
where
\[W_n:=\diag((K^2_n-(K'_n)^2)_{jj}\1(j\in J))_{1\leq j\leq d}.\]
and where
\[K''_n:=(K')^2_n+W_n=\diag((K^2_n)_{jj}\1(j\in J)+((K'_n)^2)_{jj}\1(j\in J^c))_{1\leq j\leq d}.\]
Take a further i.i.d. sequence $Z'_{pk}\sim\NNN(0,I_d),\ k=1,\ldots,m,\ p\geq1$, independent of $((S_{pk},Z_{pk}))_{p\geq1,k}$. With $K'''_n:=K^{-1}_n-(K''_n)^{-1}\geq0$ set
\[S''_{pk}:=S'_{pk}+(\lambda_{mp}K'''_n\Sigma(k/m)K'''_n)^{1/2}Z'_{pk},\quad k=1,\ldots,m,\ p\geq1.\]
Note that $\mu_{n,H}:=\LLL((Y'_{pk})_{p\geq 1,k})$ and $\mu'_{n,H}:=\LLL((S''_{pk})_{p\geq 1,k})$ are Gaussian product measures with
\[Y'_{pk}\sim\NNN(0,\lambda_{mp}K^{-1}_n\Sigma^H(k/m)K^{-1}_n+G(k/m)),\]
and
\begin{align}
\nonumber\cov(Y'_{pk})-\cov(S''_{pk})=n^{-1/4}_{\min}\lambda_{mp}\Big(&K^{-1}_nH(k/m)K^{-1}_n\\
&-(K''_n)^{-1}H(k/m)(K''_n)^{-1}\Big).\label{covdifference}
\end{align}
Set $\nu_{\max}:=\max_{1\leq j\leq d}\nu_j,\ \eta_{\min}:=\min_{1\leq j\leq d}\eta_j$ and $F'_{\max}:=\max_{1\leq j\leq d}\max_{t\in[0,1]}F'_j(t)$. Then by \eqref{lecamhell}, \eqref{hellingergauss} and \eqref{covdifference} one easily gets
\begin{equation}\label{tvbound}
\sup_H\|\mu_{n,H}-\mu''_{n,H}\|^2_{\text{TV}}\lesssim n^{-1/2}_{\min}m\xi^2_nr^2d\sum_{p\geq1}\Big(\frac{S^{-1}}{\nu_{\max}}+\frac{\eta^2_{\min}}{\lambda_{mp}n_{\min}F'_{\max}}\Big)^{-2}.
\end{equation}
The sum on the right-hand side of \eqref{tvbound} can be approximated by an integral, which is of order $\OOO(\sqrt{n_{\min}}/m)$ (by the substitution $x=pm/\sqrt{n_{\min}}$ and Remark~\ref{RmkFisherExp}). This gives $\sup_H\|\mu_{n,H}-\mu''_{n,H}\|_{\text{TV}}=\OOO(\xi_nr)$, hence
\[\delta(\ExpMSSTwoloc,\ExpMSloc)\leq\delta(\ExpMSSTwoloc,\ExpMSTwoloc)=\OOO(\xi_nr).\]
Finally, proceed analogously to obtain $\delta(\ExpMSloc,\ExpMSSTwoloc)=\OOO(\xi_nr)$.
\end{proof}
\bibliographystyle{plainnat}
\bibliography{mybib}
\end{document}